\RequirePackage{fix-cm}
\documentclass[smallextended]{svjour3}       
\smartqed  
\usepackage{graphicx}
\usepackage[%
   colorlinks=true,
   urlcolor=rltblue,       
   filecolor=rltgreen,     
   citecolor=rltgreen,     
   linkcolor=rltred,       
   bookmarks=true,
   unicode,
   plainpages=false,
   ]{hyperref}
\usepackage{xcolor}
\definecolor{rltred}{rgb}{0.75,0,0}
\definecolor{rltgreen}{rgb}{0,0.5,0}
\definecolor{rltblue}{rgb}{0,0,0.75}

\usepackage{amsmath}
\usepackage{amssymb}
\usepackage{a4wide}
\usepackage{esint}
\usepackage{bm}

\usepackage[operator,rose,frak]{paper_diening}

\DeclareMathOperator{\spt}{supp}

\newcommand{\intO}{\int_\Omega{}}
\newcommand{\T}{\bS}
\newcommand{\otimess}{\overset{s}{\otimes}}

\newcommand{\R}{\mathbb{R}}
\newcommand{\N}{\mathbb{N}}
\newcommand{\Z}{\mathbb{Z}}

\renewcommand{\div}{\operatorname{div}}
\newcommand{\Du}{\mathbf{Du}}

\newcommand{\Dw}{\mathbf{Dw}}

\newcommand{\eps}{\varepsilon}
\newcommand{\dx}{\,\mathrm{d}\mathbf{x}}
\newcommand{\dy}{\,\mathrm{d}\mathbf{y}}
\newcommand{\dz}{\,\mathrm{d}\mathbf{z}}
\newcommand{\ds}{\,\mathrm{d}s}
\newcommand{\dxi}{\,\mathrm{d}\xi}
\newcommand{\pv}{\mathrm{p.v.}}

\newcommand{\x}{\mathbf{x}}

\newcommand{\y}{\mathbf{y}}

\renewcommand{\u}{\mathbf{u}}
\renewcommand{\vv}{\mathbf{v}}
\newcommand{\w}{\mathbf{w}}

\newcommand{\f}{\mathbf{f}}
\newcommand{\z}{\mathbf{z}}
\newcommand{\xa}{\mathbf{a}}
\newcommand{\xb}{\mathbf{b}}

\newcommand{\B}{\mathcal{B}}
\newcommand{\Bf}{\mathcal{B}f}

\newcommand{\dhp}{d^+_{h,k}}
\newcommand{\dhm}{d^-_{h,k}}
\newcommand{\dhpm}{d^\pm_{h,k}}

\newcommand{\hei}{h\mathbf{e}^k}

\newcommand{\Coi}{C_0^\infty}
\newcommand{\Cooi}{C_{0,0}^\infty}

\newcommand{\supp}{\operatorname{supp}\,}
\newcommand{\dist}{\operatorname{dist}\,}
\newcommand{\diam}{\operatorname{diam}}

\newtheorem{thm}[equation]{Theorem}
\newtheorem{lem}[equation]{Lemma}
\newtheorem{prop}[equation]{Proposition}
\newtheorem{cor}[equation]{Corollary}
\newtheorem{df}[equation]{Definition}

\newtheorem{rem}[equation]{Remark}

\numberwithin{equation}{section}

\begin{document}

\title{Solenoidal difference quotients and their application to the
  regularity theory of the $p$-Stokes system}


\titlerunning{Solenoidal difference quotients and regularity of the $p$-Stokes system}        

\author{Martin K\v repela
  \and
Michael R\r u\v zi\v cka}


\institute{Martin K\r repela \at
Department of Applied Mathematics, University of  
Freiburg, Ernst-Zermelo-Stra\ss e~1, D-79104 Freiburg, Germany \\ 
              \email{martin.krepela@math.uni-freiburg.de}           
           \and
       Michael R\r u\v zi\v cka{}\at 
  Department of Applied Mathematics, University of 
  Freiburg, Ernst-Zermelo-Stra\ss e~1, D-79104 Freiburg, Germany\\
\email{rose@mathematik.uni-freiburg.de}
}

\date{Received: date / Accepted: date}

\maketitle

\begin{abstract}
  We prove existence of a~solution to the divergence equation
  satisfying a~new Bogovski-type estimate for the difference
  quotients. This enables us to give an~alternative proof of the
  interior regularity of the solution to the $p$-Stokes problem,
  completely avoiding the pressure. Moreover, as a~key preliminary
  result we prove boundedness of Calder\'on-Zygmnud operators with
  standard kernels in weighted Lebesgue and Orlicz spaces over
  a~general domain.
\end{abstract}

\section{Introduction}
We are concerned with the question of interior regularity of the weak
solution of the steady Stokes approximation for flows of shear
thinning fluids. It is given by
\begin{equation}
  \label{eq}
  \begin{aligned}
    -\divo \bfS(\bfD\bu)+\nabla \pi&=\bff\quad&&\text{in }\Omega,
    \\
    \divo\bu&=0\quad&&\text{in }\Omega,
    \\
    \bfu &= \bfzero &&\text{on } \partial \Omega,
  \end{aligned}
\end{equation}
where $\Omega \subset \R^n$, $n\ge 2$, is a~bounded domain. In here,
$\bu=(u^1,\ldots,u^n)^\top$ denotes the unknown velocity vector field,
and $\pi$ the unknown scalar pressure, while the external body force
$\bff=(f^1,\ldots,f^n)^\top$ is given. The extra stress tensor $\bfS$
depends only on $\bfD\bu:=\tfrac 12(\nabla\bu+\nabla\bu^\top)$, the
symmetric part of the velocity gradient $\nabla \bu$. The relevant
example we have in mind is 
\begin{equation}
  \label{eq:stress}
  \bfS(\bD \bu) = \mu (\delta+\abs{\bD \bu})^{p-2}
  \bD\bu\,, 
\end{equation}
with $p \in\, (1,2]$, $\delta \ge 0$, and $\mu >0$. Notice that despite
various efforts the optimal global regularity of this problem is still
open (see \cite{crispo-2009,hugo-thin-nonflat,br-reg-shearthin} for partial results). However, the interior regularity we focus on here
is well-known (see \cite{fuse,mrs,br-reg-shearthin}). The
standard proof uses localized difference quotients in each
direction. Due to the localization, the corresponding test function is
not solenoidal anymore, therefore appropriate properties of the pressure
have to be used. 

Here we modify this approach and use a~solenoidal
version of the localized difference quotient. Thus, we can completely
avoid the pressure~$\pi$. This is a~completely new approach even for the
classical Stokes problem, i.e.,~$p=2$ in \eqref{eq:stress}. To make it possible, we show that the solution of the divergence equation
obtained via the Bogovski formula (see~\cite{bo1,bo2})
satisfies an additional estimate for the difference quotient. The
proof of this is based on estimates of singular Calder\'on-Zygmund
operators generated by standard kernels in arbitrary domains $\Omega\subset \R^n$
(see~Theorem~\ref{CZ-Orlicz}). This result is of independent interest
since it shows that an~analogy of the classical
Calder\'on-Zygmund estimates (cf.~\cite{CZ56,grafakos-c}) in the whole space also
hold in arbitrary domains. To prove it, we employ ideas
from \cite{bo1,bo2,galdi-book-old-1,dr-calderon,katrin-diss}, where the divergence equation is treated, 
and modify them for our purposes.
An important feature of our result is a~careful
tracking of the dependence of the constants on various quantities,
which is missing in the literature.

\section{Preliminaries}
\subsection{Notation}
Throughout the text, we use the symbols $C,c$ to denote generic constants which may change from line
to line but are independent of ``crucial'' quantities. In many cases, the dependence of such constants on various quantities will be explicitly specified. 
Furthermore, we write $f\sim g$ if there exist constants $c,C>0$ such
that $c\, f \le g\le C\, f$. 

A~set $Q\subset\R^n$ is called an \emph{(open) cube} if there exist $a_1,\ldots,a_n\in\R$ and $\ell(Q)>0$ such that
$$
   Q = \{ \x\in\R^n \fdg a_i < x_i < a_i + \ell(Q),\, i=1,\ldots,n \}.
$$
The number $\ell(Q)$ is the \emph{side length} of $Q$
and the point
$\bc:=(a_1+ \frac 12\ell(Q), \ldots, a_n+ \frac 12\ell(Q) )^\top$ is
the center of the cube $Q$, which can be also denoted by
$Q=Q(\bc, {\ell(Q)} )$. For $\alpha>0$ and a~cube
$Q=Q(\bc, {\ell(Q)} )$ we use the notation $\alpha\, Q:=Q(\bc, \alpha {\ell(Q)})$. 
Moreover, a~cube $Q$ is said to be \emph{dyadic} if there exist
$j_i,\ldots,j_n,k\in\Z$ such that
	$$
		Q = \left\{ \x\in\R^n \fdg j_i 2^k < x_i < (j_i+1)2^k,\, i=1,\ldots,n\right \}.
	$$
Notice that, in the above definition, we consider only cubes with faces parallel to the coordinate axes.


While working with function spaces, we do not distinguish between spaces of scalar,
vector-valued or tensor-valued functions. However, we denote
vectors by boldface lower-case letters (e.g.,~$\u$) and tensors by
boldface upper-case letters (e.g.,~$\bS$). For vectors
$\u, \vv \in \R^n$, the standard tensor product $\u\otimes \vv\in \R^{n \times n}$
is defined as $(\u\otimes \vv)_{ij}:=u_iv_j$, and the symmetric tensor product as
$\u\otimess \vv:=\frac 12 (\u\otimes \vv + (\u\otimes \vv)^\top)$.
The scalar product of
vectors is denoted by $\u\cdot \vv= \sum_{i=1}^nu_i v_i$ and the
scalar product of tensors is denoted by
$\bA\cdot\bB:=\sum_{i,j=1}^nA_{i j}B_{i j}$.

If $E\subset\R^n$ is a~measurable set, $|E|$ denotes its Lebesgue measure. In the following definitions of function spaces, we always assume $E$ to be a~domain in $\R^n$, with a~sufficiently smooth boundary, if needed. We use standard Lebesgue spaces $(L^p(E),\|\cdot\|_{p})$ and Sobolev
spaces $(W^{k,p}(E),$ $\|\cdot\|_{k,p})$. If the underlying domain $E$ needs to be indicated, we denote 
the respective norms by $\|\cdot\|_{L^p(E)}$ and \mbox{$\|\cdot\|_{W^{k,p}(E)}$}. If $p\in(1,\infty)$, we denote by $p'$ the conjugated exponent $p':=\frac{p}{p-1}$.

Besides the standard $L^p$ spaces we will also consider their weighted variants. A~\emph{weight} is any measurable function $\omega:\R^n\to[0,\infty)$. If $p\in[1,\infty)$ and $\omega$ is a~weight, the space $L^p_\omega(E)$ consists of all measurable functions $f$ on $E$ such that
	$$
		\|f\|_{L^p_\omega(E)} := \left( \int_E |f(\x)|^p \omega(\x) \dx \right)^\frac1p<\infty.
	$$
If $p>1$, an~$A_p$\emph{ weight}
is a~weight such that
$\omega$ and $\omega^{1-p'}$ are locally integrable and
$$
[\omega]_{A_p} := \sup_{B\subset\R^n} \ \frac1{|B|} \int_B \omega
(\x)\dx \left( \frac1{|B|}\int_B 	\omega ^{1-p'}(\x)\dx
\right)^{p-1} < \infty, 
$$
where $B\subset\R^n$ are balls. In this case, we write $\omega\in A_p$. If $p=1$, a~weight is called an~$A_1$ \emph{weight} if $0<\omega<\infty$ a.e. and 
	$$
		[\omega]_{A_1}:= \esssup_{x\in\R^n} \frac{M\omega(\x)}{\omega(\x)} < \infty.
	$$
In here, $M$ is the \emph{maximal operator} (with respect to non-centered balls), defined by
	$$
		M\omega(\x):=\sup_{B\ni \x}\frac1{|B|} \int_B |\omega(\y)|\dy
	$$
for $\x\in\R^n$, where $B\ni \x$ means a~ball containing $\x$. Notice that $[\omega]_{A_p}\ge 1$ for all weights~$\omega$ and all $p\in[1,\infty)$, and the identity holds if and only if $\omega$ is constant (see \cite[Proposition~7.1.5]{grafakos-c}).

The symbol $\spt f$ denotes the
support of a~function $f$. 
The set of all compactly supported,
smooth functions defined on $E$ is denoted by $C^\infty_0 (E)$.
The space $W^{1,p}_0(E)$ is the closure of $C^\infty_0 (E)$ in $W^{1,p}(E)$. If $E$ is bounded, the space $W^{1,p}_0(E)$ may be equipped with the gradient norm
$\|\nabla\cdot\|_p$, thanks to the
Poincar\'e inequality. Next, we denote by $W^{1,p}_{0,\divo}(E)$ the
subspace of $W^{1,p}_{0}(E)$ consisting of solenoidal vector fields
$\u$, i.e.,~such that $\divo \u =0$. 
By $L^p_{\omega,0}(E)$ we denote the subspace of $L^p_\omega(E)$ consisting of
functions $f$ with vanishing mean value, i.e.,~such that $\int_E f(\x)\dx =0$. 

Orlicz and Sobolev--Orlicz spaces also appear frequently in this article. We briefly present their elementary properties here, for details we refer~to \cite{pkjf,ren-rao}. 

An~\emph{$N$-function} is a~continuous, nonnegative, strictly increasing and convex function $\psi$ on $[0,\infty)$ which additionally satisfies $\psi(0)=0$, $\lim_{t\to\infty} \frac{\psi(t)}t=\infty$ and $\lim_{t\to 0+} \frac{\psi(t)}t= 0$. Its \emph{conjugate $N$-function} $\psi^*:[0,\infty)\to [0,\infty)$ is then defined by
	$$
		\psi^*(t):= \sup_{s\ge 0}\ (st-\psi(s)).
	$$
Furthermore, we define
	$$
		\Delta_2(\psi):=\sup \left\{\frac{\psi(2\,t)}{\psi(t)}\, \Big|\ t\in[0,\infty)\right\}.
	$$
If $\Delta_2(\psi)<\infty$, we say that $\psi$ satisfies the \emph{$\Delta_2$-condition}.
From now on, let us assume that \mbox{$\Delta_2(\psi)<\infty$} and $\Delta_2(\psi^*)<\infty$. Then we denote by $L^\psi(E)$ and $W^{1,\psi}(E)$ the classical
Orlicz and Sobolev-Orlicz spaces, respectively. More precisely, $f \in L^\psi(E)$ if the
modular 
	$$
		\varrho_\psi(f):=\int_E \psi(|f(\x)|)\dx
	$$
is
finite, and $f \in W^{1,\psi}(E)$ if both $f$ and $ \nabla f$ belong to
$L^\psi(E)$.  When equipped with the \emph{Luxemburg norm} 
	$$
		\|f\|_{\psi}:= \inf \left\{\lambda >0\, \big|\,\int_E \psi(|f(\x)|/\lambda) \dx \le 1\right\},
	$$ 
the space $L^\psi(E)$
becomes a~Banach space. The same holds for the space
$W^{1,\psi}(E)$ when equipped with the norm $\|\cdot\|_{\psi} +\|\nabla \cdot\|_{\psi}.$
Notice that the dual space
$(L^\psi(E))^*$ can be identified with the space
$L^{\psi^*}(E)$. Furthermore, by $W^{1,\psi}_0(E)$ we denote the closure
of $C^\infty_0(E)$ in $W^{1,\psi}(E)$. If $E$ is bounded and sufficiently regular, the Poincar\'e inequality for Orlicz modulars (see \cite[Lemma 3]{talenti90}) implies that $W^{1,\psi}_0(E)$ may be equipped  with the
gradient norm $\|\nabla\cdot\|_\psi$. By $L^\psi_0(E)$ and
$C^\infty_{0,0}(E)$ we denote the subspaces of $L^\psi(E)$
and $C^\infty_0(E)$, respectively, consisting of functions $f$
such that $\int_E f(\x)\dx =0$. 

If $\psi$ is an~$N$-function satisfying $\Delta_2(\psi)<\infty$ and $\Delta_2( \psi^*)<\infty$, 
then for all $\eps >0$ there exists a~constant $c_\eps>0 $ 
such that 
\begin{equation}\label{eq:young:1}
    ts \le \eps \, \psi(t)+ c_\eps \,\psi^*(s) 
\end{equation}
holds for all $s,t\geq0$, and
\begin{equation}\label{eq:young:2}
	t\, \psi'(s) + \psi'(t)\, s \le \eps \, \psi(t)+ c_\eps \,\psi(s)
\end{equation}
holds for a.e.~$s,t>0$.

\subsection{Difference quotient}

If $k\in\{1,\ldots,n\}$, we denote by $\mathbf{e}^k$ the $k$-th vector
of the canonical basis of the Euclidean space $\R^n$, i.e.,~
$\mathbf{e}^k=(\delta_{ik})_{i=1}^n$. If $E\subset \R^n$, we denote
\begin{align*}
  E \pm \hei &:=\left\{\x \in \R^n\fdg \exists \y \in E: \x=\y
  \pm \hei\right\}, 
  \\
  E_h&:=\left\{\x \in E\fdg \dist (\x,\partial E) >h\right\}.
\end{align*}
Let $\bF:\R^n \to \R^{n\times n}$ be a~measurable tensor field (or a~vector field or a~real-valued function) and $h>0$. Then we define
the \emph{difference quotients of $\bF$} as follows: 
$$
  \dhpm \bF(\x) := \frac{\bF(\x\pm\hei)-\bF(\x)}{h}, 
\qquad \x\in\R^n.
$$
We will also use the notation $\Delta^\pm_{h,k} \bF(\x):= h\, \dhpm \bF(\x)$. 
It is well-known (cf.~\cite[Sec.~5.8]{evans-pde}) that for $ \bF \in W^{1,1}(\R^n)$ one has
\begin{align*} 
  \lim_{h\to 0+}\dhpm \bF (\x) = \partial _{k}\bF(\x) \qquad \text{for\ a.e.}\ \x\in\R^n,
\end{align*}
and
\begin{align}
  \label{eq:1a}
  \nabla \dhpm \bF (\x) = \dhpm \nabla \bF(\x) \qquad \text{for\
  a.e.}\ \x\in\R^n. 
\end{align}
Elementary calculations show that we have the following variant of the 
product rule for $\bF, \bG \in L^1_\loc(\R^n)$
\begin{align*}
  \dhpm(\bF\, \bG)(\x)= \bF(\x \pm h \be^k) \,\dhpm G(\x) + \dhpm
  \bF(\x) \,\bG(\x) \qquad \text{for\  a.e.}\ \x\in\R^n.
\end{align*}
For $\bF, \bG \in L^1(E)$, extended by zero outside $E$, the partial integration formula
\begin{align*}
  \int_E\bF \,\dhp \bG \dx = \int_E \dhm \bF\, \bG \dx
\end{align*}
holds. Moreover, for every $h_0>0$, every open set $E\subset \R^n$, every
$\bF \in W^{1,\psi}_\loc(\R^n)$ and all $h\le h_0$ we
have 
\begin{align}\label{eq:2}
  \int_{E_{h_0}}\psi (\abs{\dhpm \bF(\x)})\dx \le  \int_E\psi
  (\abs{\partial _k \bF(\x)})\dx. 
\end{align}
The proof of this assertion in the case of the special $N$-function
$\psi(t)=t^p$ can be found in \cite[Theorem 3 (i) in Section
5.8]{evans-pde}). In fact, replacing the special $N$-function $t^p$ by
a general $N$-function $\psi$ one can proceed exactly as outlined
there to obtain \eqref{eq:2}. These comments also apply for the
converse statement \cite[Theorem 3 (ii) in Section 5.8]{evans-pde}).
More precisely, if $\dhpm \bF \in L^\psi(E_{h_0})$ for all $h_0>0$ and all
$0<h<h_0$ satisfy
\begin{align}\label{eq:2a}
  \int_{E_{h_0}}\psi (\abs{\dhpm \bF(\x)})\dx \le c_1,
\end{align}
then $\partial _k \bF$ exists in the sense of distributions
and satisfies
\begin{align}\label{eq:2b}
  \int_{E}\psi  (\abs{\partial _k \bF(\x)})\dx\le c_1.
\end{align}
For the convenience of the reader we give a proof of these results in
the appendix. 

\subsection{Operators with kernels}

A~measurable function $K:\R^n\times\R^n \to\R$ is called a~\emph{kernel}. 
\begin{df}
  Let
  $E\subset\R^n$ be a~domain. A~kernel $K:\R^n\times\R^n\to \R$ is called a~\emph{standard
    kernel with respect to $E$} if there exists a~constant $\kappa_1>0$
  such that, for any $\x,\y,\z\in E$ satisfying
  \begin{equation}\label{pulka}
    \x\ne\y \quad \text{and}\quad |\x-\z|\le \frac12|\x-\y|,
  \end{equation}
  the following conditions hold true:
  \begin{align}
    |K(\x,\y)| &\le \frac{\kappa_1}{|\x-\y|^n}, \label{SK1}\\
    |K(\x,\y)-K(\z,\y)| &\le \kappa_1\,\frac{ |\x-\z|}{ |\x-\y|^{n+1}}, \label{SK2}\\
    |K(\y,\x)-K(\y,\z)| &\le \kappa_1\,\frac{|\x-\z|}{ |\x-\y|^{n+1}}. \label{SK3}
  \end{align}
  The set of all kernels satisfying the above conditions with a~given constant $\kappa_1$ will be denoted by $\mathrm{SK}(E,\kappa_1)$.
\end{df}

\begin{rem}
  As it is common in the literature, it is possible to replace
  the right-hand side in \eqref{SK2} and \eqref{SK3} by
  $$
      \kappa_1\,\frac{|\x-\z|^\delta}{ |\x-\y|^{n+\delta}}
  $$
  with $\delta>0$. Although the results could be obtained in this
  generalized setting as well, we restrict ourselves to the case
  $\delta=1$ to avoid further complications. 
\end{rem}

\begin{df}
We say that a~linear operator $T$ on $C_0^\infty (\R^n)$ is 
\emph{generated} by the kernel $K$ if 
\begin{align}
   Tf(\x)& = \int_{\R^n} f(\y)K(\x,\y) \dy \label{ass-op}
\end{align}
holds whenever the right-hand side is well-defined. 
For a~given  kernel $K$ and $\eps>0$ we define the \emph{truncated kernel} $K_\eps$ by 
\begin{align*}
  K_\eps(\x,\y) &:=
  \begin{cases}
    K(\x,\y) &\text{for } \abs{\x-\y}>\eps,
    \\
    0 &\text{for } \abs{\x-\y} \leq \eps,
  \end{cases}
\end{align*}
and denote by $T_\eps$ the operator generated by the kernel $K_\eps$.
\end{df}

\begin{df}
  Let $K$ be a~kernel and let $E\subset\R^n$ be a~domain. We
  call 
  $K$ a~\emph{Calder\'on-Zygmund kernel with respect to $E$}, if there exists a~constant $\kappa_2\in(0,\infty)$ such that the
  function $N\colon \R^n \times \R^n \to \R$, defined by
  $$
     N(\x,\y) := K(\x,\x-\y), \qquad \x,\y\in\R^n,
  $$
  satisfies the following conditions:
  \begin{gather}
   N(\x,\alpha\z)  = \alpha^{-n} N(\x,\z) \text{\quad for all } \x,\z\in\R^n,\ \alpha>0; \label{CZ1} \\
   \int_{|\z|=1} N(\x,\z) \dz = 0 \text{\quad for all } \x\in\R^n;  \label{CZ2} \\
   \sup_{\substack{\x\in E}} \int_{|\z|=1} |N(\x,\z)|^2 \dz  \le \kappa_2^2. \label{CZ3}
 \end{gather}
The set of all kernels satisfying the above conditions with a~given constant $\kappa_2$ will be denoted by $\mathrm{CZ}(E,\kappa_2)$.
\end{df}
\begin{rem}
  One can replace \eqref{CZ3} by
  $$
       \sup_{\substack{\x\in E}} \int_{|\z|=1} |N(\x,\z)|^q \dz  \le \kappa_2^q, 
  $$
  where $q \in (1,\infty)$, and still retain the relevant properties of the operator generated by the kernel. For the sake of simplicity, we use the condition only with $q=2$.  
\end{rem}

\subsection{$(p,\delta)$-structure}
\label{sec:stress_tensor}
Let us define what it means that a~tensor field $\bS$ has
a~$(p,\delta)$-structure. For details on this matter,
see~\cite{die-ett,dr-nafsa}. For a~tensor
$\bfP \in \setR^{n \times n} $ we denote its symmetric part by
	$$
		\bP^\sym:= \frac{\bP +\bP^\top}2 \in \R_\sym ^{n \times n},
	$$
where
	$$
		\R_\sym ^{n \times n} := \set {\bfP \in \setR^{n \times n} \,|\, \bP =\bP^\top}.
	$$
We use the notation $\abs{\bP}^2=\bP \cdot \bP$.

It is convenient to define a~special $N$-function
$\phi=\phi_{p,\delta}$, with $p \in (1,\infty)$, $\delta\ge0$, by
\begin{equation} 
  \label{eq:5}
  \varphi(t):= \int _0^t (\delta +s)^{p-2} s \ds, \qquad t\ge 0.
\end{equation}
The function $\phi$ satisfies, uniformly in $t$ and independently of
$\delta$, the equivalences
\begin{equation*}
  \label{eq:equi1}
  \phi''(t)\, t \sim \phi'(t),
  \qquad \phi'(t)\, t \sim \phi(t),
  \qquad
  t^p+\delta^p \sim \phi(t) +\delta^p. 
\end{equation*}
In case that the (one-sided) derivative $\phi''(0)$ does not exist, we assume that $\phi''(t)\, t$ is continuously extended by zero for $t= 0$. 
We define the \emph{shifted
  $N$-functions} $\set{\phi_a}_{a \ge 0}$
(cf. ~\cite{die-ett,die-kreu,dr-nafsa}) by
\begin{equation*}
  \phi_a(t):= \int _0^t \frac { \phi'(a+s)\,s}{a+s} \ds, \qquad t\ge0. 
\end{equation*}
Note that the family $\set{\phi_a}_{a \ge 0}$ satisfies the $\Delta_2$-condition
uniformly with respect to ${a \ge 0}$, i.e.,
	$$
		\sup_{a\ge0}\ \sup_{t\ge 0} \
                \frac{\phi_a(2t)}{\phi_a(t)} \le c\, 2^{\max\{2,p\}}.
	$$
	
\begin{df} 
\label{ass:1}
  Let $\bS\colon \setR^{n \times n} \to \setR^{n \times n}_\sym $ be a~tensor field  
  satisfying 
	 $$
		 \bS \in C^0(\setR^{n \times n},\setR^{n \times n}_\sym )\cap C^1(\setR^{n \times n}\setminus \{\bfzero\}, \setR^{n \times n}_\sym ), 
	$$
 $\bS(\bP) = \bS\big (\bP^\sym \big )$ whenever $P\in\R^{n\times n}$, and $\bS(\mathbf 0)=\mathbf 0$. 
 We say that $\bS$ has a~\emph{$(p,\delta)$-structure} if for some $p \in (1, \infty)$,
  $\delta\in [0,\infty)$, and the $N$-function
  $\varphi=\varphi_{p,\delta}$ (cf.~\eqref{eq:5}) there exist
  constants $\gamma_0, \gamma_1 >0$ such that the inequalities
   \begin{equation*}
     \begin{aligned}
       \sum\limits_{i,j,k,l=1}^n \partial_{kl} S_{ij} (\bP)
       Q_{ij}Q_{kl} &\ge \gamma_0
       \,\phi''(|\bP^\sym|)|\bQ^\sym|^2\,,
       \\[-2mm]
       \big |\partial_{kl} S_{ij}({\bP})\big | &\le \gamma_1 \,\phi''
       (|\bP^\sym|)
     \end{aligned}
   \end{equation*} 
   are satisfied for all $\bP,\bQ \in \setR^{n\times n} $ with
   $\bP^\sym \neq \bfzero$ and all $i,j,k,l=1,\ldots,n$.  The constants
   $\gamma_0$, $\gamma_1$, and $p$ are called the {\em
     characteristics} of $\bfS$.
\end{df}
\begin{rem}
  %
    An important example of a~tensor field $\bfS$ having
    a~$(p,\delta)$-structure is given by $ \bfS(\bfP) =
    \phi'(\abs{\bfP^\sym})\abs{\bfP^\sym}^{-1} \bfP^\sym$.  In this
    case, the characteristics of $\bfS$, namely $\gamma_0$ and
    $\gamma_1$, depend only on $p$ and are independent of $\delta \geq
    0$.  
\end{rem}
Suppose that a~tensor field $\bS$ has a~$(p,\delta)$-structure. Then we define its \emph{associated tensor field}
 $\bF\colon\setR^{n \times n} \to \setR^{n \times
  n}_\sym$ by
\begin{align}
  \label{eq:def_F}
  \bF(\bP):= \big (\delta+\abs{\bP^\sym} \big )^{\frac
    {p-2}{2}}{\bP^\sym } .
\end{align}
The connection between $\bfS$, $\bfF$ and $\set{\phi_a}_{a \geq 0}$
is best explained by the following proposition
(cf.~\cite{die-ett,dr-nafsa}). 
\begin{prop}
  \label{lem:hammer}
  Let $\bfS$ have a~$(p,\delta)$-structure, and let $\bfF$ be defined
  in~\eqref{eq:def_F}.  Then
    \begin{align}
      \notag 
      \big({\bfS}(\bfP) - {\bfS}(\bfQ)\big) \cdot \big(\bfP-\bfQ \big)
      &\sim \bigabs{ \bfF(\bfP) - \bfF(\bfQ)}^2\,,
      \\
      &\sim \phi_{\abs{\bfP^\sym}}(\abs{\bfP^\sym - \bfQ^\sym})\,,
      \notag 
      \\
      \notag 
      &\sim \phi''\big( \abs{\bfP^\sym} + \abs{\bfP^\sym - \bfQ^\sym}
      \big)\abs{\bfP^\sym - \bfQ^\sym}^2\,,
      \\
      \abs{\bfS(\bfP) - \bfS(\bfQ)} &\sim
      \phi'_{\abs{\bfP^\sym}}\big(\abs{\bfP^\sym -
        \bfQ^\sym}\big)\,,  \label{eq:hammere}  
      \intertext{uniformly in $\bfP, \bfQ \in \setR^{n \times n}$.
        Moreover,  uniformly in $\bfQ \in \setR^{n \times n}$,} 
      \notag 
      \bfS(\bfQ) \cdot \bfQ &\sim \abs{\bfF(\bfQ)}^2 \sim
      \phi(\abs{\bfQ^\sym}).
    \end{align}
  The constants depend only on the characteristics of $\bfS$.
\end{prop} 
For a~detailed discussion of the properties of $\bS$ and $\bF$ and their
relation to Orlicz spaces and $N$-functions we refer the reader
to~\cite{dr-nafsa,bdr-7-5}. 
In what follows, we shall work only with $\bS (\bP)$ and $\bF(\bP)$, where $\bP$ is 
a~symmetric tensor. Therefore, we can drop the superscript
``$^\sym $'' in the above formulas.

If $\bS$ has a~$(p,\delta)$-structure, from
Proposition~\ref{lem:hammer} we easily obtain the following equivalences:
\begin{align*}
  \begin{aligned}
    \notag 
    \abs{\dhpm \T (\Du)} &\sim
    (\delta+|\Du|+|\Delta^\pm_{h,k}{\Du}|)^{p-2}|\dhpm \Du|
    \\
    &\sim \phi''(|\Du|+|\Delta^\pm_{h,k}{\Du}|) |\dhpm \Du|
    \\
    &\sim \big (\phi''(|\Du|+|\Delta^\pm_{h,k}{\Du}|) \big)^\frac
    12\abs{\dhpm \bF (\Du)}\,
    \\
    & \sim (\delta+|\Du|+|\Delta^\pm_{h,k}{\Du}|)^{\frac {p-2}2}|\dhpm
    \bF(\Du)|\,,
  \end{aligned}
\end{align*}    
\vspace{-3mm}
\begin{align}
  \hspace*{-9.7mm}
  \begin{aligned}
    \label{eq:3a}
    \dhpm \T (\Du)\cdot \dhpm \Du &\sim \abs{\dhpm \bF (\Du)}^2
    \\
    &\sim (\delta+|\Du|+|\Delta^\pm_{h,k}\Du|)^{p-2}|\dhpm \Du|^2
    \\
    &\sim \phi''(|\Du|+|\Delta^\pm _{h,k}{\Du}|) |\dhpm \Du|^2.
  \end{aligned}
\end{align}
The equivalence constants depend here only on the characteristics of $\bS$.
All assertions from this section may be proved by easy manipulations
of definitions, and we omit their proofs.


\section {Calder\'on-Zygmund estimates}

Our interest lies in estimates concerning Orlicz modulars. However, we are going to prove the results first in weighted $L^p$ spaces. The following known extrapolation principle (see \cite[Theorem 4.15]{CUMP11}) offers an~elegant connection between the two settings.
\begin{thm}\label{extra}
	Let $p_0\in[1,\infty)$ and let $\mathcal F$ be a~family of pairs of nonnegative measurable functions on $\R^n$. Suppose that there exists a~nondecreasing function $C_0:[1,\infty)\to(0,\infty)$ such that the inequality
		\begin{equation*} 
			\left( \int_{\R^n} f^{p_0}(\x)\omega(\x)\dx \right)^\frac1{p_0} \le C_0\big([\omega]_{A_{p_0}}\big) \left(\int_{\R^n} g^{p_0}(\x)\omega(\x)\dx\right)^\frac1{p_0}
		\end{equation*}
	holds for all $(f,g)\in\mathcal F$ and all weights $\omega\in A_{p_0}$. Then for every $N$-function $\psi$ satisfying $\Delta_2(\psi)<\infty$ and $\Delta_2(\psi^*)<\infty$ there exists a~constant $C\in(0,\infty)$ such that the inequality
		$$
			\int_{\R^n} \psi(f(\x)) \dx \le C \int_{\R^n} \psi(g(\x)) \dx
		$$
	holds for all $(f,g)\in\mathcal F$.
\end{thm}

\begin{rem}
	(i) The assumption on monotonicity of $C_0$ with respect to $[\omega]_{A_{p_0}}$ is necessary for the proof of Theorem \ref{extra}.
	 However, this is rarely explicitly mentioned in the literature. 
	 
	(ii) Careful tracking of the constants in the proof of
        \cite[Theorem 4.15]{CUMP11} reveals that the constant $C$ in
        Theorem \ref{extra} depends on $C_0$, $\Delta_2(\psi)$ and
        $\Delta_2(\psi^*)$.  To check this, one may use the inequality
        $[\omega_1 \omega_2^{-1}]_{A_2}\le [\omega_1]_{A_1}
        [\omega_2]_{A_1}$ (see \cite[Proposition 3.5]{CUMP11}) as well
        as the boundedness of the maximal operator in Orlicz spaces
        (see~\cite[Theorem 2.2]{gallardo}). The latter may be
        expressed by the inequality
        \begin{align*}
          \int_{\R^n} \psi(|{M}f(\x)|)\dx \le \Delta_2(\psi^*) \int_{\R^n} \psi(2c^2|f(\x)|)\dx,
        \end{align*}
        where $c$ is the weak $(1,1)$-constant of the maximal
        operator, which depends on $n$ only. 
        In fact, one can show that
        \begin{align*}
        C= \widetilde c\, C_0\big (2^{p_0}
        \rho_\psi(M)^{p_0-1}\rho_{\psi^*}(M)\big )^{\alpha}.
        \end{align*}
        In here, $\widetilde c$ is a~constant depending on
        $\Delta_2(\psi), \Delta_2(\psi^*)$, the exponent $\alpha$
        depends on $\Delta_2(\psi)$, and $\rho_\psi(M)$ is the
        constant in the modular estimate of the maximal operator (it
        can be estimated by $\Delta_2(\psi^*)\,\Delta_2(\psi)^{\beta}$, where $\beta$
        depends on $n$).
\end{rem}

The auxiliary result below is a~simple version of the Whitney covering lemma (cf.~\cite[p.~348]{BS88}).
\begin{prop}\label{Whitney}
	Let $\Omega\in\R^n$ be a~domain. Then there exists a~sequence $\{Q_j\}_{j\in\N}$ of dyadic cubes satisfying:
		\begin{itemize}
			\item[\rm(i)]
				$\Omega=\bigcup_{j\in\N} \overline{Q}_j,$
		 	\item[\rm(ii)]
			 	$Q_j\cap Q_k \neq \varnothing \Leftrightarrow j=k,$
			\item[\rm(iii)]
				$\dist(Q_j,\partial\Omega)>4 \diam Q_j \text{ for all } j\in\N.$
		\end{itemize}
\end{prop}
\begin{proof}
	Let $\left\{\x^k\right\}_{k\in\N}$ be a~dense sequence of points in $\Omega$. Define $Q_1$ as the largest\footnote{Such cube exists since if a~point is contained in two dyadic cubes, then one of the cubes is a~subset of the other one.} dyadic cube $Q$ such that $\x^1\in \overline{Q}$ and $\dist(Q,\partial\Omega)>4 \diam Q$. Next, suppose that $Q_1,\ldots,Q_{j-1}$ are defined and let $k\in\N$ be the smallest index such that $\x^k\notin\bigcup_{i=1}^{j-1}Q_i$. Then define $Q_j$ as the largest dyadic cube $Q$ such that $\x^k\in \overline{Q}$ and $\dist(Q,\partial\Omega)>4 \diam Q$. The sequence $\left\{Q_j\right\}_{j\in\N}$ obviously satisfies (ii) and (iii). To verify (i), let $\x\in\Omega$ be arbitrary. There exists a~dyadic cube $Q$ such that $\x\in \overline{Q}$ and $\dist(Q,\partial\Omega)>4 \diam Q$. By density, $\x^k\in Q$ for some $k\in\N$. Then necessarily $\overline{Q}\subseteq \overline{Q}_j$ for some $j\in\N$, hence $\x\in\bigcup_{j\in\N}\overline{Q}_j$.
\end{proof}

The next theorem justifies the definition of Calder\'on-Zygmund operators with standard kernels by the Cauchy principal value of the integral \eqref{ass-op} and shows their boundedness in $L^p_\omega(\Omega)$ and $L^\psi(\Omega)$. Although the result is well-known, in standard literature it appears only in the setting $\Omega=\R^n$. We prove it below for any domain $\Omega\subset\R^n$.
\begin{thm}\label{CZ-Orlicz}
  Let $\Omega\subset\R^n$ be a~domain. Let $K\in \mathrm{SK}(\Omega,\kappa_1)\cap \mathrm{CZ}(\Omega,\kappa_2)$ for some $\kappa_1,\kappa_2\in(0,\infty)$. Then the following assertions hold true:
  
  {\rm(i)}  Let $T_\eps $ be the operators generated by
  the truncated kernels $K_\eps$.  Then for any $f\in\Coi(\Omega)$ the
  expressions
 \begin{align}
	     Tf(\x) & := \lim_{\eps \to 0} T_\eps f (\x) , \label{opT}\\
   T^{(*)}f(\x) & := \sup_{\eps>0} |T_\eps f(\x)|     \notag 
 \end{align}
 are defined for a.e.~$\x\in\Omega$. 

 {\rm (ii)} For every $p \in (1,\infty)$ there exists a~nondecreasing function $C_p:[1,\infty)\to (0,\infty)$ such that the operator $T$ defined by \eqref{opT} admits, for each $\omega \in A_p$, a~bounded extension to $L^p_{\omega}(\Omega)$ satisfying 
 \begin{align}\label{modular-p}
   \begin{aligned}
     \left( \int_{\Omega} |Tf(\x)|^p \omega(\x)\dx \right)^\frac 1p &\le C_p\big( [\omega]_{A_p} \big)(\kappa_1+\kappa_2) \left( \int_{\Omega}|f(\x)|^p\omega(\x)\dx \right)^\frac 1p
   \end{aligned}
 \end{align}
 for all $f\in L^p_\omega(\Omega)$. An~analogous assertion holds for the operator $T^{(*)}$.

 {\rm (iii)} For every $N$-function $\psi$  with $\Delta_2(\psi), \Delta_2(\psi^*)<\infty$ there
 exists a~positive constant $C_\psi$ such that 
 the operator $T$ defined by \eqref{opT} admits a~bounded extension to $L^\psi(\Omega)$ satisfying 
 \begin{align}\label{modular-psi}
   \begin{aligned}
     \int_{\Omega} \psi(|Tf(\x)|)\dx &\le C_\psi \int_{\Omega}
     \psi((\kappa_1+\kappa_2)|f(\x)|)\dx 
   \end{aligned}
 \end{align}
 for all $f\in L^\psi(\Omega)$. An~analogous assertion holds for the operator $T^{(*)}$.
\end{thm}

\begin{rem}\label{ind-Omega}
	 In the above theorem, $C_p$ and $C_\psi$ are independent of the domain $\Omega$ in the following sense.
	 Let $\Omega,\Omega'\subset \R^n$ be domains, let $K$, $K'$ be kernels such that $K\in \mathrm{SK}(\Omega,\kappa_1)\cap \mathrm{CZ}(\Omega,\kappa_2)$, $K'\in \mathrm{SK}(\Omega',\kappa_1)\cap \mathrm{CZ}(\Omega',\kappa_2)$ with the same $\kappa_1,\kappa_2\in(0,\infty)$, and let $T$, $T'$ be the corresponding operators. 
	 Then, for each $p\in(1,\infty)$, inequality \eqref{modular-p} holds with the same $C_p$ when $\Omega$ is replaced by $\Omega'$ and $T$ by $T'$. An~analogous assertion holds regarding \eqref{modular-psi}.
\end{rem}

\begin{proof}[of Theorem \ref{CZ-Orlicz}]
  We are going to prove the theorem only for the operator $T$.   The proof for $T^{(*)}$ follows the same reasoning.
  At first, let us assume that 
	\begin{equation}\label{vetsi}
		  K\in \mathrm{SK}(\widetilde{\Omega},\kappa_1)\cap \mathrm{CZ}(\widetilde{\Omega},\kappa_2),
	\end{equation}
	where
    $$
	    \widetilde{\Omega}:=\{\x\in\R^n \,|\, \dist(\x,\Omega) \le 4\diam \Omega\}.
    $$
  Let $\eta\in\Coi(\R^n)$ satisfy $0\le\eta\le 1$,
  $\eta\big|_\Omega\equiv 1$, $\eta(\x)=0$ for all $\x\in\R^n$
  such that $\dist(\x,\Omega)>\diam\Omega$, and $\|\nabla\eta\|_\infty \le 2(\diam\Omega)^{-1}$. For $\x,\y\in\R^n$ define
  $$
   \widehat  K(\x,\y):=K(\x,\y)\eta(\x), \qquad \text{and} \qquad \widetilde{K}(\x,\y):=K(\x,\y)\eta(\x)\eta(\y).
  $$		
  Then the kernel
  $$
   \widehat N(\x,\y):=\widehat K(\x,\x-\y)=K(\x,\x-\y)\eta(\x) =N(\x,\y)\eta(\x)
  $$
  satisfies the conditions \eqref{CZ1} and \eqref{CZ2}. Moreover, it satisfies
  \eqref{CZ3} globally, i.e.,~with $E=\R^n$. Thus, $\widehat  K \in\mathrm{CZ}(\R^n,\kappa_2)$. Therefore, by
  \cite[Theorem 2]{CZ56},
  $$
    \widehat Tg(\x) :=\lim _{\eps \to 0} \int_{\R^n} g(\y) \widehat
    K_{\eps}(\x,\y)\dy=\lim_{\eps \to 0} \widehat T_\varepsilon g(\x)
  $$
  is defined for a.e.~$\x\in\R^n$ and all $g\in\Coi(\R^n)$ and it
  admits a~bounded extension from $L^2(\R^n)$ to $L^2(\R^n)$. It
  follows from the proofs and comments in \cite{CZ56}, in particular
  p.~295 and a~remark on p.~306, that the corresponding operator norm
  satisfies \mbox{$\|\widehat T\|_{L^2\to L^2}\le c\,\kappa_2,$} where
  \mbox{$c>0$} is a~fixed constant depending only on the
  dimension $n$. 
  Since $\widetilde{T}_\eps g(\x)= \widehat T_\eps(g\eta)(\x)$ holds
  for all $\eps>0$, $g\in\Coi(\R^n)$ and a.e.~$\x\in\R^n$, we may draw
  the same conclusions about the operator
 $$
   \widetilde{T}g(\x) := \lim _{\eps \to 0} \int_{\R^n}
   g(\y)\widetilde{K}_\eps (\x,\y) \dy= \lim_{\eps \to 0} \widetilde T_\eps g(\x).
 $$ 
 In particular, we get   $\widetilde{T} g(\x)= \widehat
 T(g\eta)(\x)$, and thus 
	 \begin{equation}\label{el2}
		 \|\widetilde{T}\|_{L^2\to L^2} \le \|\widehat{T}\|_{L^2\to L^2}\le c\,\kappa_2.
	 \end{equation}
 Let us next show that $\widetilde{K}$ is a~standard
 kernel with respect to $\R^n$. Suppose that \mbox{$\x,\y,\z\in\R^n$} satisfy \eqref{pulka}. Then
		\begin{equation}\label{123}
			|\x-\y|\le 2|\z-\y| \le 3|\x-\y|.
		\end{equation}
 We will distinguish two cases. At first, assume that 
		\begin{equation}\label{omezeni}
			\x,\z\in \widetilde\Omega \quad \text{and }\ \dist(\y,\Omega)\le\diam\Omega.
		\end{equation}
 Then obviously $y\in \widetilde\Omega$, and since $K\in \mathrm{SK}(\widetilde\Omega,\kappa_1)$, we have  
 \begin{align*} 
   |\widetilde{K}(\x,\y)| & = |K(\x,\y)\eta(\x)\eta(\y)| \le \frac{\kappa_1}{|\x-\y|^n}, \\
   |\widetilde{K}(\x,\y)-\widetilde{K}(\z,\y)| & = |K(\x,\y)\eta(\x)\eta(\y)-K(\z,\y)\eta(\z)\eta(\y)| \\
                          & \le |K(\x,\y)-K(\z,\y)||\eta(\x)\eta(\y)| + |K(\z,\y)||\eta(\x)-\eta(\z)||\eta(\y)| \\
                          & \le (1+ 3\cdot 2^{n+2}) \kappa_1 \frac{|\x-\z|}{ |\x-\y|^{n+1}}, \\
   |\widetilde{K}(\y,\x)-\widetilde{K}(\y,\z)|	& \le |K(\y,\x)-K(\y,\z)||\eta(\x)\eta(\y)|	+ |K(\y,\z)||\eta(\x)-\eta(\z)||\eta(\y)| \\
                          & \le (1+3\cdot 2^{n+2}) \kappa_1 \frac{|\x-\z|}{ |\x-\y|^{n+1}}.			   
 \end{align*}
 In the second and third condition we used the estimate
 $$
   \frac{|\x-\z|}{|\z-\y|^n} = \frac{|\x-\z||\x-\y|}{|\z-\y|^n|\x-\y|}
   \le \frac{3\cdot 2^{n+1}\diam\Omega\, |\x-\z|}{|\x-\y|^{n+1}},
 $$
	 which holds due to \eqref{123} and \eqref{omezeni}, and the property $\|\nabla\eta\|_\infty \le 2(\diam\Omega)^{-1}$. 
	
 Now suppose that \eqref{pulka} holds but \eqref{omezeni} does not. Then
 one of the following situations occurs:
	
 (i) $\dist(\y,\Omega)>\diam\Omega$, in which case $\y\notin\supp \eta$.
 
 (ii) $\dist(\y,\Omega)\le\diam\Omega$ and $\dist(\x,\Omega)>4\diam\Omega$. Then $\x\notin\supp\eta$ and, moreover, we get
 \begin{align*}
   \dist(\x,\Omega) & \le \dist(\z,\Omega) + |\x-\z| \le \dist(\z,\Omega) + \frac12 |\x-\y| \\
                    & \le \dist(\z,\Omega) + \frac12 \left( \dist(\x,\Omega) + \dist(\y,\Omega) + \diam\Omega\right) \\
                    & \le \dist(\z,\Omega) + \frac12 \dist(\x,\Omega) + \diam\Omega.
 \end{align*}
 Hence, 
 $$
   \dist(\z,\Omega) \ge \frac12 \dist(\x,\Omega) - \diam\Omega > \diam\Omega,
 $$
 which shows that $\z\notin\supp\eta$.
	
 (iii) $\dist(\y,\Omega)\le\diam\Omega$ and $\dist(\z,\Omega)>4\diam\Omega$. Then $\z\notin\supp\eta$ and, moreover, we get
 \begin{align*}
   \dist(\z,\Omega) & \le \dist(\x,\Omega) + |\x-\z| \le \dist(\x,\Omega) + \frac12 |\x-\y| \\
                    & \le \dist(\x,\Omega) + \frac12 \left( \dist(\x,\Omega) + \dist(\y,\Omega) + \diam\Omega\right) \\
                    & \le \frac32 \dist(\x,\Omega) + \diam\Omega.
 \end{align*}
 Hence, 
 $$
    \dist(\x,\Omega) \ge 2 \diam\Omega,
 $$
 which shows that $\x\notin\supp\eta$.	
 
 In each of the three cases we get
 $$
   \widetilde{K}(\x,\y)=\widetilde{K}(\z,\y)=\widetilde{K}(\y,\x)=\widetilde{K}(\y,\z)=0,
 $$
 hence the conditions \eqref{SK1}--\eqref{SK3} are satisfied trivially.
 Altogether, we have now verified that $\widetilde{K} \in \mathrm{SK}(\R^n,\kappa_0)$ with the constant
 $
   \kappa_0=: (1+3\cdot 2^{n+2}) \kappa_1. 
 $
 Thanks to this and the estimate \eqref{el2}, \cite[Theorem 7.4.6]{grafakos-c} yields that
 for every $p\in(1,\infty)$ there exists a~nondecreasing function $C_p:[1,\infty)\to(0,\infty)$ such that
 \begin{equation}\label{eq:pw}
   \left( \int_{\R^n}|\widetilde{T}f(\x)|^p\omega(\x)\dx \right)^\frac 1p \le C_p\big( [\omega]_{A_p} \big) (\kappa_1+\kappa_2)
   \left( \int_{\R^n}|f(\x)|^p\omega(\x)\dx \right)^{\frac 1p}
 \end{equation}
 holds for all $f\in \Coi(\R^n)$ and all $\omega\in A_p$. It has to be noted that the desired monotone dependence of $C_p$ on $[\omega]_{A_p}$ is not mentioned in \cite{grafakos-c} but it can be verified through a~careful inspection of the proofs leading to the result in there.
 Notice also that $C_p$ is independent of $\kappa_1$, $\kappa_2$ and $\Omega$.
 
 In the next step, let $\psi$ be an~$N$-function with
 $\Delta_2(\psi)<\infty$ and $\Delta_2(\psi^*)<\infty$. Due to
 \eqref{eq:pw} Theorem \ref{extra}
 provides the existence of a~positive constant $C_\psi$ such that 
 \begin{equation}\label{modRn}
   \int_{\R^n} \psi(|\widetilde{T}f(\x)|)\dx \le C_\psi \int_{\R^n} \psi((\kappa_1+\kappa_2)|f(\x)|)\dx
 \end{equation}
 for all $f\in\Coi(\R^n)$. To obtain this estimate we have performed the extrapolation with respect to the family
	 $$
		 \mathcal F := \left\{ \big(\widetilde{T}f,(\kappa_1+\kappa_2)f\big) \,\big|\, f\in\Coi(\R^n) \right\}.
	 $$
 The obtained constant $C_\psi$ is independent of $\kappa_1$, $\kappa_2$ and $\Omega$. For any $f\in\Coi(\Omega)$ and $\x\in\Omega$ we have
 $$
   \widetilde{T}f(\x) = \eta(\x)Tf(\x) = Tf(\x).
 $$
 Hence, \eqref{modRn} and \eqref{eq:pw} yield \eqref{modular-p} and \eqref{modular-psi}, respectively, for $f\in\Coi(\Omega)$. Since $\Coi(\Omega)$ is dense in $L^p_\omega(\Omega)$ as well as in $L^\psi(\Omega)$, the operator $T$  admits corresponding bounded extensions such that \eqref{modular-p} and \eqref{modular-psi} hold for any $f\in L^p_\omega(\Omega)$ and $f\in L^\psi(\Omega)$, respectively. 
  
 So far, we have proved the theorem under the stronger assumption \eqref{vetsi}. To prove it in full generality, let $\Omega\subset\R^n$ be an~arbitrary domain and $K\in \mathrm{SK}(\Omega,\kappa_1)\cap \mathrm{CZ}(\Omega,\kappa_2)$. By Proposition \ref{Whitney} there exists a~sequence $\{Q_j\}_{j\in\N}$ of pairwise disjoint dyadic cubes such that $\Omega=\bigcup_{j\in\N} \overline{Q}_j$ and 
	 $$
		 \widetilde{Q}_j := \{ \x\in\R^n \fdg \dist(\x,Q_j)\le 4\diam Q_j \} \subset \Omega
	 $$
 for all $j\in\N$. Hence, for any $j\in\N$, we have $K\in \mathrm{SK}(\widetilde{Q}_j,\kappa_1)\cap \mathrm{CZ}(\widetilde{Q}_j,\kappa_2)$, and using the first part of the proof we get
	 \begin{align*}
		 \int_{\Omega} |Tf(\x)|^p \omega(\x)\dx & = \sum_{j\in\N} \int_{Q_j} |Tf(\x)|^p \omega(\x)\dx \\
												& \le C_p^p\big( [\omega]_{A_p} \big) (\kappa_1+\kappa_2)^p \sum_{j\in\N} \int_{Q_j} |f(\x)|^p\omega(\x)\dx\\
												& = C_p^p\big( [\omega]_{A_p} \big) (\kappa_1+\kappa_2)^p \int_{\Omega} |f(\x)|^p\omega(\x)\dx
	 \end{align*}
 for all $\omega\in A_p$ and $f\in L^p_\omega(\Omega)$. In here, it is important that $C_p^p\big( [\omega]_{A_p} \big)$ does not depend on~$Q_j$. Similarly, we obtain \eqref{modular-psi} for any $f\in L^\psi(\Omega)$. Note also that Remark \ref{ind-Omega} is justified.
\end{proof}

\section{Divergence equation}

We proceed with proving the new estimate concerning the Bogovski solution of the divergence equation. We prove the result in the weighted-$L^p$ setting. The variant for Orlicz modulars will be then obtained as a~corollary by extrapolation.

\begin{thm}\label{difLp}
	Let $Q\subset \R^n$ be an~open cube. There exists a~linear operator $\B:\Cooi(Q)\to W^{1,\infty}_0(Q)$ with the following properties:
	\begin{itemize}
		\item[\rm(i)] 
                 For every $f\in\Cooi(Q)$, the equation
		\begin{equation}\label{diveq}
		\div (\Bf) = f 
		\end{equation}
		is satisfied in $Q$. 
		\item[\rm(ii)]
		For every $p\in(1,\infty)$ there exists a~nondecreasing function $C_{p}:[1,\infty)\to(0,\infty)$ such that
		\begin{equation}\label{bogest-p}
		\left( \int_Q |\nabla(\Bf)(\x)|^p \omega(\x) \dx \right)^\frac 1p \le C_{p} \big( [\omega]_{A_p} \big) \left( \int_Q |f(\x)|^p \omega(\x)\dx \right)^\frac 1p.
		\end{equation}
		holds for all $\omega\in A_p$ and all $f\in\Cooi(Q)$. The function $C_p$ is independent of $Q$.
		\item[\rm(iii)] 
		For every $p\in(1,\infty)$ there exists a~nondecreasing function $\widetilde{C}_{p}:[1,\infty)\to(0,\infty)$ such that
		\begin{align}\label{R2-p}
                  \begin{aligned}
                    &\left( \int_Q \big|\dhpm \nabla(\Bf)(\x)\big|^p \omega(\x) \dx \right)^\frac 1p \\
                    &\le \widetilde{C}_{p} \big( [\omega]_{A_p} \big) \left(
                      \int_{Q} \left( \big|\dhp f(\x)\big| + \big|\dhm
                        f(\x)\big| + \frac{|f(\x)|}{\ell(Q)} \right)^p
                      \omega(\x) \dx \right)^\frac 1p
                  \end{aligned}
		\end{align}
		holds for all $h>0$, $\omega\in A_p$ and $f\in\Cooi(Q)$. The function $\widetilde{C}_p$ is independent of $Q$.
		\item[\rm(iv)]
		For every $p\in(1,\infty)$ and $\omega\in A_p$ there exists a~continuous extension of $\B$ to $L^p_{\omega,0}(Q)$ such that \eqref{diveq}, \eqref{bogest-p} and \eqref{R2-p} hold for all $f\in L^p_{\omega,0}(Q)$. In here, $f$ and $\Bf$ are extended by zero outside $Q$ so that their difference quotients are defined a.e.~in $\R^n$.
	\end{itemize}
\end{thm}

\begin{cor}\label{difOrlicz}
  Let $Q\subset \R^n$ be an~open cube. Let
  $\psi$ be an~$N$-function with $\Delta_2(\psi)<\infty$ and $\Delta_2(\psi^*)<\infty$. 
  Then there exists a~linear operator $\B:L_0^\psi(Q) \to W_0^{1,\psi}(Q)$ with the following properties:
  \begin{itemize}
  \item[\rm(i)] Equation \eqref{diveq} is satisfied a.e.~in $Q$ for
    all $f \in L^\psi_0(Q)$.
  \item[\rm(ii)]
    There exists a~positive constant $C_\psi$ such that 
    \begin{equation*} 
      \int_Q \psi(|\nabla(\Bf)(\x)|)\dx \le C_\psi \int_Q \psi(|f(\x)|)\dx
    \end{equation*}
	 holds for all $f\in L^\psi_0(Q)$. The constant $C_\psi$ is independent of $Q$.
\item[\rm(iii)] There exists a~positive constant $\widetilde{C}_\psi$ such that 
      \begin{equation*} 
      \int_Q \psi \left(\big|\dhpm \nabla(\Bf)(\x)\big| \right)\dx 
      \le \widetilde{C}_\psi  \int_{Q} \psi \left( \big|\dhp f(\x)\big| + \big|\dhm f(\x)\big| + \frac{|f(\x)|}{\ell(Q)}
      \right) \dx 
    \end{equation*}
  holds for all $h>0$ and all $f\in L^\psi_0(Q)$. In here, $f$ and $\Bf$ are extended by zero outside $Q$ so that their difference quotients are defined a.e.~in $\R^n$. The constant $\widetilde{C}_\psi$ is independent of $Q$.
  \end{itemize}	
\end{cor}

Before we prove Theorem \ref{difLp}, let us show the following simple technical proposition.
\begin{lem}\label{2dif}
	Let $\varrho\in C^\infty(\R^n,\R)$ and $\xa,\xb,\w\in\R^n$. Then
	\begin{equation}\label{trivka}
	|\varrho(\xa)-\varrho(\xa+\w)-\varrho(\xb)+\varrho(\xb+\w)| \le \|\nabla^2\varrho\|_\infty |\w| |\xa-\xb|.
	\end{equation}
\end{lem}
\begin{proof}
	We have
		$$
			|\varrho(\xa)-\varrho(\xa+\w)-\varrho(\xb)+\varrho(\xb+\w)|  \le \|\nabla\varrho-\nabla\varrho(\,\cdot + \w)\|_\infty|\xa-\xb| \le \|\nabla^2\varrho\|_\infty|\w||\xa-\xb|.
		$$
\end{proof}

\begin{proof}[of Theorem \ref{difLp}]
  At first suppose that $Q=\left(-\frac12,\frac12\right)^n$. Choose a~fixed function $\varrho\in C_0^\infty(\R^n)$ such that $\supp \varrho\subset \left(-\frac14,\frac14\right)^n$
  and $\int_Q\varrho(\x)\dx = 1$. For any $f \in \Cooi(Q)$ define 
		\begin{equation}\label{defB}
			  \Bf(\x):= \int_Q f(\y) \left( \frac{\x-\y}{|\x-\y|^n}  \int_{|\x-\y|}^\infty \varrho\left( \y + \xi \frac{\x-\y}{|\x-\y|}\right) \xi^{n-1} \dxi \right) \dy. 
		\end{equation}
From \cite[Chapter III, Lemma 3.1]{galdi-book-old-1} it follows that $\Bf\in W^{1,\infty}_0(Q)$ and it satisfies \eqref{diveq} in $Q$. 

From now on, we will use the notation 
\begin{align*}
\pv \int_E g(\y)K(\x,\y)\dy :=\lim_{\eps \to  0} \int_{E \cap \{ |\x-\y|>\eps \}} g(\y)K(\x,\y) \dy
\end{align*}
for functions $g \in \Coi(\R^n)$ and kernels $K$ whenever the limit on
the right-hand side exists.

Part (ii) is proved in \cite{katrin-diss,john} without the monotone
dependence of $C_p$ on $[\omega]_{A_p}$. Note that this dependence of
the constant is mentioned in \cite{katrin-paper}. For the sake of the
self-consistency of this paper we sketch its proof here. 
The technique used here is the same as in part (iii), where
full details will be given.  Let $i,j\in\{1,\ldots,n\}$. Since $Q$ is
convex and it contains $\supp \varrho$, by
\cite[p.~119]{galdi-book-old-1} we have
	\begin{align*}
	\partial_j \Bf_i(\x) 
	& = \pv\int_{Q}  f(\y) J_{ij}(\x,\y)\dy 
		 + \int_{Q} f(\y) U_{ij}(\x,\y) \dy \\
	& \qquad + f(\x) \int_{Q} \frac{(x_i-y_i)(x_j-y_j)}{|\x-\y|^2} \varrho(\y) \dy
	\end{align*} 
for any $f\in\Cooi(Q)$. In here,
	$$
		J_{ij}(\x,\y) 
		 := \frac{\delta_{ij}}{|\x-\y|^n} \int_0^\infty	\!\! \varrho \left( \x +  \xi\frac{\x-\y}{|\x-\y|}	\right) \xi^{n-1} \dxi 
		 + \frac{x_i-y_i}{|\x-\y|^{n+1}} \int_0^\infty \!\! \partial_j\varrho \left( \x + \xi \frac{\x-\y}{|\x-\y|} \right) \xi^n \dxi,
	$$
and $U_{ij}$ is a~kernel such that $|U_{ij}(\x,\y)| \le c\,\|\varrho\|_{{1,\infty}}|\x-\y|^{1-n}$, where $c$ is a~positive constant. From \cite[Lemma 6.1]{dr-calderon} we get that
$J_{ij}\in \mathrm{SK}(Q,\kappa)\cap \mathrm{CZ}(Q,\kappa)$, where $\kappa$ depends only on $\|\varrho\|_{{1,\infty}}$ and $n$. By Theorem \ref{CZ-Orlicz}
there exists a~positive nondecreasing function $c_p$ such that 
	\begin{equation}\label{jedna}
		\left( \int_Q \left| \pv\int_{Q}  f(\y) J_{ij}(\x,\y)\dy \right|^p \omega(\x) \dx \right)^\frac1p \le c_p \big( [\omega]_{A_p} \big) \left( \int_Q |f(\x)|^p \omega(\x) \dx \right)^\frac1p
	\end{equation}
for all $f\in\Cooi(Q)$. Furthermore, we have (see \cite[p.~218]{dr-calderon})
	$$
		\int_{Q} f(\y) U_{ij}(\x,\y) \dy \le 2^nc\,\|\varrho\|_{{1,\infty}} Mf(\x)
	$$
for $f\in\Cooi(Q)$ and $\x\in Q$. Hence, by \cite[Theorem 7.1.9(b)]{grafakos-c} there exists a~positive constant $C$ (possibly depending on $p$, $n$) such that 
	\begin{equation}\label{dva}
		\left( \int_Q \left| \int_{Q} f(\y) U_{ij}(\x,\y) \dy \right|^p \omega(\x) \dx \right)^\frac1p \le C[\omega]_{A_p}^{\frac 1{p-1}}\|\varrho\|_{{1,\infty}} \left( \int_Q |f(\x)|^p \omega(\x) \dx \right)^\frac 1p
	\end{equation}
for all $f\in\Cooi(Q)$. Obviously, we also have
	\begin{equation}\label{tri}
		\left( \int_Q \left| f(\x) \int_{Q} \frac{(x_i-y_i)(x_j-y_j)}{|\x-\y|^2} \varrho(\y) \dy \right|^p \omega(\x) \dx \right)^\frac1p
		\le \|\varrho\|_{1} \left( \int_Q |f(\x)|^p \omega(\x) \dx \right)^\frac 1p.
	\end{equation}
Therefore, by combining \eqref{jedna}, \eqref{dva} and \eqref{tri}, we have shown that there exists a~positive nondecreasing function $C_p$ such that \eqref{bogest-p} holds for all $f\in\Cooi(Q)$. Recall however that so far we have assumed $Q=\left(-\frac12,\frac12\right)^n$. The result for a~general cube will be obtained by rescaling at the end of the proof of part (iii).

We continue with part (iii).		
We are going to prove it only for the difference quotient $\dhp$. The proof for $\dhm$ is fully analogous.

We are still assuming that $Q=\left(-\frac12,\frac12\right)^n$ and $\varrho$ is as above. Moreover, suppose that $f\in \Cooi(Q)$, $\x\in Q$ and $h>0$. 
We get the following representation:
\allowdisplaybreaks
\begin{align*}
   \dhp(\Bf)(\x) 
  & = \frac1h \int_Q f(\y)  \frac{ \x + \hei - \y }{ |\x+\hei-\y|^n }
    \int_{|\x+\hei-\y|}^\infty \varrho\left( \y + \xi \frac{ \x + \hei
    - \y }{ |\x+\hei-\y| } \right) \xi^{n-1} \dxi \dy 
  \\ 
  & \quad - \frac1h \int_Q f(\y)  \frac{ \x - \y }{ |\x-\y|^n }
    \int_{|\x-\y|}^\infty \varrho\left( \y + \xi \frac{ \x - \y }{
    |\x-\y| } \right) \xi^{n-1} \dxi \dy 
\\
  & = \frac1h \int_{Q-\hei} f(\y+\hei)  \frac{ \x - \y }{ |\x-\y|^n }
    \int_{|\x-\y|}^\infty \varrho\left( \y + \hei + \xi \frac{ \x - \y
    }{ |\x-\y| } \right) \xi^{n-1} \dxi \dy 
\\ 
  & \quad - \frac1h \int_Q f(\y)  \frac{ \x - \y }{ |\x-\y|^n }
    \int_{|\x-\y|}^\infty \varrho\left( \y + \xi \frac{ \x - \y }{
    |\x-\y| } \right) \xi^{n-1} \dxi \dy 
\\ 
 & = \int_{Q-\hei} \dhp f(\y)   \frac{ \x - \y }{
   |\x-\y|^n } \int_{|\x-\y|}^\infty \varrho\left( \y + \hei + \xi
   \frac{ \x - \y }{ |\x-\y| } \right) \xi^{n-1} \dxi \dy 
\\ 
  & \quad + \frac1h \int_{(Q-\hei)\setminus Q} f(\y)  \frac{ \x - \y
    }{ |\x-\y|^n } \int_{|\x-\y|}^\infty \varrho\left( \y + \hei + \xi
    \frac{ \x - \y }{ |\x-\y| } \right) \xi^{n-1} \dxi \dy 
\\ 
  & \quad - \frac1h \int_{Q\setminus (Q-\hei)} f(\y)  \frac{ \x - \y
    }{ |\x-\y|^n } \int_{|\x-\y|}^\infty \varrho\left( \y + \hei + \xi
    \frac{ \x - \y }{ |\x-\y| } \right) \xi^{n-1} \dxi \dy 
\\ 
  & \quad + \int_Q f(\y)  \frac{ \x - \y }{ |\x-\y|^n }
    \int_{|\x-\y|}^\infty \dhp\varrho\left( \y + \xi \frac{ \x - \y }{ |\x-\y| } \right) \xi^{n-1} \dxi \dy 
\\ 
  & = \int_{Q\cup (Q-\hei)} \dhp f(\y) \frac{ \x -
    \y }{ |\x-\y|^n } \int_{|\x-\y|}^\infty \varrho\left( \y + \hei +
    \xi \frac{ \x - \y }{ |\x-\y| } \right) \xi^{n-1} \dxi \dy 
\\ 
  & \quad + \int_Q f(\y)  \frac{ \x - \y }{ |\x-\y|^n }
    \int_{|\x-\y|}^\infty \dhp \varrho\left( \y + \xi \frac{ \x - \y
    }{ |\x-\y| } \right)\xi^{n-1} \dxi \dy
\\ 
  & =: \mathbf{a}^{h,1}(\x) + \mathbf{a}^{h,2}(\x).
\end{align*}
To get the fourth equality we used the fact that $\supp f \subset Q$. 
A~simple observation yields
	\begin{equation}\label{ind-h}
		\|\varrho(\cdot+\hei)\|_{{1,\infty}} = \|\varrho\|_{{1,\infty}}  \quad \text{for all }h>0.
	\end{equation}
Let $i,j\in\{1,\ldots,n\}$ and assume that $h\in\left(0,\frac14\right)$. 
Then the set $Q\cup(Q-\hei)$ is convex and it contains $\supp\varrho(\cdot+\hei)$. Hence, we may use
\cite[p.~119]{galdi-book-old-1} to get the following representation:
\begin{align*}
  \partial_j a^{h,1}_i(\x) 
  & = \pv\int_{Q\cup (Q-\hei)} \dhp f(\y) K^h_{ij}(\x,\y)\dy \\
  & \qquad + \int_{Q\cup (Q-\hei)} \dhp f(\y) G^h_{ij}(\x,\y) \dy \\
  & \qquad + \dhp f(\x) \int_{Q\cup (Q-\hei)} \frac{(x_i-y_i)(x_j-y_j)}{|\x-\y|^2} \varrho(\y+\hei) \dy.\\
\end{align*}
In here, $G_{ij}$ is a~kernel such that 
	$ |G^h_{ij}(\x,\y)| \le c\, \|\varrho\|_{{1,\infty}} |\x-\y|^{1-n}$ for all $\x,\y\in Q\cup(Q-\hei)$, $\x\ne\y$.
The positive constant $c$ here depends only on $n$. We have also used \eqref{ind-h} here. Furthermore, the kernel $K^h_{i,j}$ is expressed as
\begin{align}
  K^h_{ij}(\x,\y) 
  & := \frac{\delta_{ij}}{|\x-\y|^n} \int_0^\infty
    \varrho \left( \x + \hei + \xi\frac{\x-\y}{|\x-\y|}
    \right) \xi^{n-1} \dxi \nonumber \\ 
  & \quad + \frac{x_i-y_i}{|\x-\y|^{n+1}}
    \int_0^\infty \partial_j\varrho \left( \x + \hei +
    \xi \frac{\x-\y}{|\x-\y|} \right) \xi^n
    \dxi\nonumber
\end{align}
for any $\x,\y\in\R^n$, $\x\ne\y$. 
From the restriction $h\in\left(0,\frac14\right)$ it follows that 
	$$
		Q\cup (Q-h\be_k) \subset E := {Q\cup \left(Q-\textstyle\frac14\be_k\right)}.
	$$	
In \cite[Lemma 6.1]{dr-calderon} it is
proved that there exist constants $\kappa_1,\kappa_2\in(0,\infty)$ such that for every $h>0$ it holds that $K^h_{ij}\in \mathrm{SK}(E,\kappa_1) \cap \mathrm{CZ}(E,\kappa_2)$. The constants $\kappa_1$, $\kappa_2$ depend on $n$ and $\|\varrho(\cdot+\hei)\|_{{1,\infty}}$. 
From \eqref{ind-h} it follows that $\kappa_1$, $\kappa_2$ are again independent of $h$. 

Theorem~\ref{CZ-Orlicz} now grants the existence of a~positive nondecreasing function $\widetilde{c}_{p,1}$ such that for all $f\in\Cooi(Q)$ and all $h\in\left(0,\frac14\right)$ the inequality
\begin{align*}
  &\left( \int_{Q} \left| \pv\int_{Q\cup (Q-\hei)} \dhp f(\y)
    K^h_{ij}(\x,\y)\dy \right|^p \omega(\x) \dx \right)^\frac 1p
  \\
  &\le \widetilde{c}_{p,1} \big( [\omega]_{A_p} \big) \left( \int_{Q\cup
    (Q-\hei)} \big|\dhp f(\x)\big|^p \omega(\x)  \dx \right)^\frac
    1p. 
\end{align*}
holds true. In the same way as in part (ii), there exists a~positive constant $C$ such that
\begin{align*}
  &\left( \int_{Q} \left| \int_{Q\cup (Q-\hei)} \dhp f(\y)
    G^h_{ij}(\x,\y)\dy \right|^p \omega(\x) \dx \right)^\frac 1p
  \\ 
  &\le C[\omega]^\frac1{p-1}_{A_p} \|\varrho\|_{{1,\infty}} \left(
    \int_{Q\cup (Q-\hei)} \big| \dhp f(\x) \big|^p \omega(\x) \dx
    \right)^\frac 1p 
\end{align*}
and 
\begin{align*}
  &\left( \int_{Q} \left|\dhp f(\y)\int_{Q\cup (Q-\hei)}
    \frac{(x_i-y_i)(x_j-y_j)}{|\x-\y|^2} \varrho(\y + \hei) \dy
    \right|^p \omega(\x) \dx \right)^\frac 1p
  \\ 
  &\le \|\varrho\|_{1} \left( \int_{Q} \big|\dhp f(\x)\big|^p
    \omega(\x) \dx \right)^\frac 1p  
\end{align*}
both hold for all $f\in\Cooi(Q)$ and all $h\in\left(0,\frac14\right)$. It follows from the obtained estimates that there exists a~positive nondecreasing function $\widetilde{C}_{p,1}$ such that
\begin{equation}\label{cast1}
  \left( \int_{Q} \left| \partial_j a^{h,1}_i(\x) \right|^p \omega(\x) \dx \right)^\frac 1p 
  \le \widetilde{C}_{p,1} \big( [\omega]_{A_p} \big) \left( \int_{Q\cup (Q-\hei)} \big| \dhp f(\x)\big|^p \omega(\x) \dx \right)^\frac 1p
\end{equation}
holds for all $f\in\Cooi(Q)$ and all $h\in\left(0,\frac14\right)$. 

Let us proceed by estimating the partial derivatives of $\mathbf{a}^2(\x)$, still under the assumption \mbox{$h\in\left(0,\frac14\right)$}. 
Since $\supp \varrho\subset \left(-\frac14,\frac14\right)^n$ holds, the
function $\dhp \varrho$ is supported in $Q$.
The cube $Q$ is convex, thus we may again use the same calculation as in \cite[p.~119]{galdi-book-old-1} to obtain
\begin{align}
  \partial_j a^{h,2}_i(\x) 
  & = \pv\int_{Q} f(\y) M^h_{ij}(\x,\y)\dy 
   + \int_{Q} f(\y) H^h_{ij}(\x,\y) \dy \label{dz2}\\
  & \qquad + f(\x) \int_{Q} \frac{(x_i-y_i)(x_j-y_j)}{|\x-\y|^2}
    \dhp\varrho(\y) \dy.\nonumber
\end{align}
In here, $H^h_{ij}$ is a~kernel satisfying 
$
  |H^h_{ij}(\x,\y)|\le C|\x-\y|^{1-n}
$
for all $\x,\y\in Q$, $\x\ne\y$, with $C$ depending only on $n$ and $\|\varrho\|_{{2,\infty}}$.
Furthermore, the kernel $M^h_{ij}$ is, for $\x,\y\in \R^n$, $\x\ne\y$,
defined as follows:
\begin{align*}
  M^h_{ij}(\x,\y) 
  & := \frac{\delta_{ij}}{|\x-\y|^n} \int_0^\infty
    \dhp\varrho\left( \x + \xi \frac{ \x - \y }{ |\x-\y| }\right) \xi^{n-1} \dxi \\ 
  & \qquad + \frac{x_i-y_i}{|\x-\y|^{n+1}} \int_0^\infty
    \dhp\partial_j\varrho\left( \x + \xi \frac{ \x - \y }{ |\x-\y| }\right) \xi^{n} \dxi.\nonumber\\ 
  & =: m^{ h,1}_{ij}(\x,\y) + m^{h,2}_{ij}(\x,\y).
\end{align*}
Observe that if $\x,\y\in Q$ and $\xi>n$, then
	$$
		\xi - h > n - \textstyle\frac14 > \frac{\sqrt{n}}4 = \diam\,(\supp \varrho),
	$$
and therefore
	$$
		\x + \xi \frac{ \x - \y }{ |\x-\y| } \notin \supp\varrho \quad \text{and} \quad  \x + \hei + \xi \frac{ \x - \y }{ |\x-\y|} \notin \supp\varrho.
	$$
Hence, the integrals over $(0,\infty)$ in
the definition of $M^h_{ij}$ may be replaced by integrals over
$(0,n)$.

From \cite[Lemma 6.1]{dr-calderon} it follows that $m^{h,1}_{ij}$ is a~Calder\'on-Zygmund kernel with respect to $Q$, with a~constant independent of $h$. Here, notice that the function $\dhp\varrho$ plays the role of $\varrho$ in \cite{dr-calderon}, and $\|\dhp\varrho\|_{{1,\infty}}\le \|\varrho\|_{{2,\infty}}$. In the next step, we shall verify that $m^{h,1}_{ij}$ is a~standard kernel with respect to $Q$ with a~constant independent of $h$. Let $\x,\y,\z\in Q$ be such that $\x\ne\y$ and $|\x-\z|\le \frac12|\x-\y|$. It is easy to see that 
	$$
		 |m^1_{ij}(\x,\y)| \le \frac{\|\nabla\varrho\|_\infty}{|\x-\y|^n} \int_0^{n} \xi^{n-1}\dxi = \frac{\|\nabla\varrho\|_\infty n^{n-1}}{|\x-\y|^n}.
	$$
Using \eqref{123}, we obtain
	\begin{equation}\label{dif01}
		\left|\frac1{|\x-\y|^n}-\frac1{|\y-\z|^n}\right| = \frac{\big||\y-\z|-|\x-\y|\big|}{|\x-\y|^n|\y-\z|^n} \cdot \sum_{b=0}^{n-1}|\y-\z|^b|\x-\y|^{n-b-1} \le  \frac{2^nn|\x-\z|}{|\x-\y|^{n+1}}
	\end{equation} 
and 
	\begin{align}
		\left| \frac{x_i-y_i}{|\x-\y|^{n+1}} - \frac{z_i-y_i}{|\z-\y|^{n+1}} \right| & 
		\le \frac{|x_i-z_i|}{|\x-\y|^{n+1}} + |y_i-z_i| \left| \frac1{|\x-\y|^{n+1}} - \frac1{|\z-\y|^{n+1}} \right| \nonumber\\ 
		& \le \left( 1 + 3\cdot 2^n (n+1) \right) \frac{|\x-\z|}{|\x-\y|^{n+1}}. \label{dif03} 
	\end{align}	
Moreover, we have
	\begin{equation}\label{dif02}
		\left| \frac{\x-\y}{|\x-\y|} - \frac{\z-\y}{|\z-\y|} \right| = \left| \frac{\x-\z}{|\x-\y|} + (\z-\y)\frac{|\z-\y|-|\x-\y|}{|\x-\y||\z-\y|} \right| \le 2 \frac{|\x-\z|}{|\x-\y|}.
	\end{equation}	
Now we use \eqref{dif01}, \eqref{dif02}, Lemma \ref{2dif} and the inequality 
	$$
		|\x-\z| \le \frac{\diam Q\,|\x-\z|}{|\x-\y|} = \frac{2\sqrt{n}\,|\x-\z|}{|\x-\y|}
	$$
to get 
	 \begin{align*}
		 & |m^{h,1}_{ij}(\x,\y)-m^{h,1}_{ij}(\z,\y)| \\
		 & \quad\le \left|\frac1{|\x-\y|^n}-\frac1{|\y-\z|^n}\right| \int_0^n \left| \dhp\varrho\left( \z + \xi \frac{ \z - \y }{ |\z-\y| } \right) \right| \xi^{n-1} \dxi\\	
		 & \quad\qquad + \frac1{|\x-\y|^n} \int_0^n \left| \dhp\varrho\left( \x + \xi \frac{ \x - \y }{ |\x-\y|} \right) - \dhp\varrho\left( \z + \xi \frac{ \z - \y }{ |\z-\y|} \right) \right| \xi^{n-1} \dxi\\
		 & \quad \le (2n)^n \|\nabla\varrho\|_\infty \frac{|\x-\z|}{|\x-\y|^{n+1}} + \frac{\|\nabla^2\varrho\|_\infty}{|\x-\y|^n} \int_0^n \left( |\x-\z| + 2\xi\frac{|\x-\z|}{|\x-\y|} \right) \xi^{n-1}\dxi\\
		 & \quad \le ( 2^n + 4 )n^n\|\varrho\|_{{2,\infty}} \frac{|\x-\z|}{|\x-\y|^{n+1}}.
	\end{align*}
Analogously, we obtain the following:
	\begin{align*}
		& |m^{h,1}_{ij}(\y,\x)-m^{h.1}_{ij}(\y,\z)| \\
		& \quad \le \left|\frac1{|\x-\y|^n}-\frac1{|\y-\z|^n}\right| \int_0^{n} \left| \dhp\varrho\left( \y + \xi \frac{ \y - \z }{ |\y-\z| } \right) \right| \xi^{n-1} \dxi\\	
		& \quad \qquad + \frac1{|\x-\y|^n} \int_0^{n} \left| \dhp\varrho\left( \y + \xi \frac{ \y - \x }{ |\y-\x| } \right) - \dhp\varrho\left( \y + \xi \frac{ \y - \z }{ |\y-\z|} \right) \right| \xi^{n-1} \dxi\\
		& \quad \le 2^{n+1}n^{n} \|\varrho\|_{{2,\infty}} \frac{|\x-\z|}{|\x-\y|^{n+1}}.
	\end{align*}
Therefore, $m^{h,1}_{ij}$ is a~standard kernel with respect to $Q$ with a~constant independent of $h$. Analogously, using  \eqref{123}, \eqref{dif03} and \eqref{dif02}, we show that $m^{h,2}_{ij}$ is a~standard kernel with respect to $Q$ with a~constant depending only on $n$ and $\|\varrho\|_{{3,\infty}}$. Hence, $M^{h}_{ij}$ is a~standard kernel with respect to $Q$ with a~constant independent of $h$.
By Theorem \ref{CZ-Orlicz} there exists a~positive nondecreasing function $\widetilde{c}_{p,2}$ such that for all $h\in\left(0,\frac14\right)$ and all $f\in\Cooi(Q)$ we have
		$$
			\left( \int_{Q} \left| \pv\int_{Q} f(\y) M^h_{ij}(\x,\y)\dy  \right|^p \omega(\x) \dx \right)^\frac1p
			\le \widetilde{c}_{p,2} \big( [\omega]_{A_p} \big) \left( \int_{Q} \left|f(\x)\right|^p \omega(\x) \dx \right)^\frac1p.
		$$
The remaining parts of \eqref{dz2} are treated analogously as their counterparts in $\frac{\partial a^1_j}{\partial x_i}$. Altogether, it follows that there exists a~positive nondecreasing function $\widetilde{C}_{p,2}$ such that
	$$
		\left( \int_{Q} \left| \frac{ \partial}{\partial x_j} a^{h,2}_i(\x) \right|^p \omega(\x) \dx \right)^\frac1p  
		\le \widetilde{C}_{p,2} \big( [\omega]_{A_p} \big) \left( \int_{Q} |f(\x)|^p \omega(\x) \dx \right)^\frac1p
	$$
holds for all $h\in(0,\frac14)$ and $f\in\Cooi(Q)$. Using this result and \eqref{cast1}, we obtain the existence of a~positive nondecreasing function $\widetilde{C}_p$ such that the following holds for all $h\in\left(0,\frac14\right)$ and all $f\in\Cooi(Q)$:
	\begin{align*}
		\left(  \int_Q \big|\dhp\nabla(\Bf)(\x)\big|^p \omega(\x)\dx \right)^\frac1p \!\!
		& \le \widetilde{C}_{p} \big( [\omega]_{A_p} \big) \left( \int_{Q} \left( \big| \nabla \mathbf{a}^{h,1}(\x) \big| + \big| \nabla \mathbf{a}^{h,2}(\x) \big|\right)^p \omega(\x) \dx \right)^\frac1p\\
		& \le \widetilde{C}_{p} \big( [\omega]_{A_p} \big) \left( \int_{Q\cup (Q-\hei)} \left( \big| \dhp f(\x) \big| + |f(\x)| \right)^p \omega(\x) \dx \right)^\frac1p\\
		& \le \widetilde{C}_{p} \big( [\omega]_{A_p} \big) \left( \int_{Q} \! \left( \big| \dhp f(\x) \big| \! + \! \big| \dhm f(\x) \big| \! + \! |f(\x)| \right)^p \! \omega(\x) \dx \right)^\frac1p\!\!.
	\end{align*}
As next, assume that $h\in\left[\frac14,\infty\right).$ Let $C_p$ be the function obtained in part (ii). By \eqref{bogest-p} we have
	\begin{align*}
		\left(  \int_Q \big|\dhp\nabla(\Bf)(\x)\big|^p \omega(\x) \dx \right)^\frac1p \! \! & \le C_{p} \big( [\omega]_{A_p} \big) \left( \int_Q \left|\frac{\nabla(\Bf)(\x+\hei)}h-\frac{\nabla(\Bf)(\x)}h\right|^p \omega(\x) \dx \right)^\frac1p \\
		& \le 4C_{p} \big( [\omega]_{A_p} \big) \left( \int_Q  \!\! \left( |\nabla(\Bf)(\x\!+\!\hei)| \! + \! |\nabla(\Bf)(\x)| \right)^p \!\omega(\x) \dx \! \right)^\frac1p\\
		& \le 8C_{p} \big( [\omega]_{A_p} \big) \left( \int_Q |\nabla(\Bf)(\x)|^p \omega(\x) \dx \right)^\frac1p\\
		& \le 8C_{p} \big( [\omega]_{A_p} \big) \left( \int_Q |f(\x)|^p \omega(\x) \dx \right)^\frac1p
	\end{align*}
for all $f\in\Cooi(Q)$ and $h\ge\frac14$. In the third step we have used the fact that $\nabla(\Bf)(\x)=0$ for $\x\in\R^n\setminus Q$.

At this point we have proved that if $Q=\left(-\frac12,\frac12\right)^n$, then there exist positive nondecreasing functions $C_p$, $\widetilde{C}_p$ such that for all $f\in\Cooi(Q)$ and all $h>0$ the inequalities \eqref{bogest-p} and
	$$
		\left( \int_Q \big|\dhp\nabla(\Bf)(\x)\big|^p \omega(\x) \dx \right)^\frac1p \!\! \le \widetilde{C}_p \big( [\omega]_{A_p} \big) \! \left( \int_{Q} \!\! \left( \big| \dhp f(\x) \big| \!+ \!\big| \dhm f(\x) \big| \!+\! |f(\x)| \right)^p \!\omega(\x) \dx \right)^\frac1p
	$$
are satisfied. In here, the operator $\B$ is defined by \eqref{defB}. Now consider the cube $\lambda Q=\left(-\frac\lambda2,\frac\lambda2\right)^n$ with a~fixed $\lambda>0$. Define $\varrho_\lambda(\x):=\lambda^{-n}\varrho(\x)$ for $\x\in\lambda Q$. Then $\varrho_\lambda\in\Coi(\lambda Q)$ and $\int_{\lambda Q} \varrho_\lambda(\x)\dx = 1$. 
For any $f\in\Cooi(\lambda Q)$ define $\B_\lambda f$ by the formula \eqref{defB}, where $Q$ and $\varrho$ are replaced by $\lambda Q$ and $\varrho_\lambda$, respectively. Similarly as before one verifies that $\div \B_\lambda f= f$ holds a.e.~in $\lambda Q$. 
Using the proven result for $Q$ and rescaling through the change of variables $\x\to\lambda\x$, we obtain
	$$  
		 \left(\int_{\lambda Q} |\nabla(\B_\lambda f)(\x)|^p \omega(\x) \dx \right)^\frac1p \le C_p \big( [\omega]_{A_p} \big) \left( \int_{\lambda Q} |f(\x)|^p \omega(\x) \dx \right)^\frac1p
	$$
and
	$$
		 \left(\int_{\lambda Q} \!\big|\dhpm \!\nabla(\B_\lambda f)(\x)\big|^p \omega(\x) \dx \right)^\frac1p \!\!\!\!\le \!\widetilde{C}_p \big( [\omega]_{A_p} \big) \!\!\left( \int_{\lambda Q} \!\!\left( \big|\dhp f(\x)\big| \!+\! \big|\dhm f(\x)\big| \!+\! \frac{|f(\x)|}{\lambda} \right)^p \!\! \omega(\x) \dx \right)^\frac1p\!\!. 
	$$
Notice also that $\lambda=\ell(\lambda Q)$. Obviously, this result remains unchanged when the cube $\lambda Q$ is shifted, e.g.,~replaced by $\x_0+\lambda Q$ for any fixed $\x_0\in\R^n$. Hence, we have now proved (i)--(iii) in full generality.

Part (iv) is proven by a~standard approximation argument, using density of $\Cooi(Q)$ in $L^p_{\omega,0}(Q)$ for any fixed $p$ and $\omega$. This completes the proof.
\end{proof}

\begin{rem}
	Corollary \ref{difOrlicz} follows now from Theorems \ref{difLp} and \ref{extra}.
\end{rem}

\section{Interior Regularity}

From now on, let $\Omega \subset \R^n$ be a~bounded domain. We also
assume that $\bS$ has \mbox{a~$(p,\delta)$-structure} for some
$p \in (1,\infty)$ and $\delta \ge 0$.  The properties of $\bS$ and
the standard theory of monotone operators imply in a~standard way  the
existence of a~unique $\u\in W^{1,p}_{0,\divo}(\Omega)$, satisfying
\begin{equation} \label{eq:rovnice}
  \int_\Omega\bS(\Du)\cdot\Dw \dx=\intO \f \cdot \w \dx
\end{equation}
for all $\bw\in W^{1,p}_{0,\divo}(\Omega)$, i.e., $\bu$ is a weak
solution of \eqref{eq}. By choosing $\bw=\bu$ in \eqref{eq:rovnice} and
using Proposition~\ref{lem:hammer}, the properties of $\bS$, the
Poincar\'e inequality (cf.~\cite[Lemma 3]{talenti90}), Korn inequality
(cf.~\cite[Theorem~6.10]{john}) and Young inequality
\eqref{eq:young:1}, we obtain that this solution satisfies the
a~priori estimate
\begin{align*}
  \gamma_0(p)\intO \varphi(|\nabla\u|)  \dx \le c\, \intO \phi^*(\abs{\f})\dx .
\end{align*}
From here on, we denote by $\gamma_i(p)$, $i=0,1$, the constants in
Definition~\ref{ass:1} for a~given~$p$. Moreover, all
constants may depend on the characteristics of $\bS$,
$\diam (\Omega)$, $\abs{\Omega}$, the space dimension $n$ and on the
John constants of $\Omega$.
The dependence on these quantities will not be mentioned explicitly anymore. However, dependence on
other quantities will be specified.

\begin{thm}
  \label{thm:MT1}
  Let the extra stress tensor $\bS$ have
  a~$(p,\delta)$-structure for some $p \in (1,2]$ and
  $\delta \in [0,\infty)$, and let $\bF$ be the associated tensor
  field to $\bS$. Let $\Omega \subset\R^n$ be a~bounded domain 
  and let $\f\in L^{\phi ^*}(\Omega)$. Let $Q \subseteq \Omega$ be a~cube and let $\xi \in \Coi(\Omega)$ satisfy
  $\chi_{\frac 12 Q} \le \xi \le \chi_{\frac 34Q}$. Then the unique
  weak solution $\u\in W^{1,p}_{0,\divo}(\Omega)$ of the
  problem~\eqref{eq:rovnice} satisfies
  \begin{equation}
    \begin{aligned}
      \label{eq:est-reg-F2}
      \int_{\frac 12 Q}   \abs{\nabla \bF(\Du)}^2\dx
          \le  c(\norm{\xi}_{2,\infty}) \int_Q {\phi^*}(|\f|)
          +\phi(\abs {\nabla \bu})\dx .
    \end{aligned}
  \end{equation}  
\end{thm}

\begin{proof}
  Let $k\in\{1,\ldots,n\}$. We would like to use
  $\w=\dhm\big(\xi^2 \, \dhp\u\big)$ as a~test function in
  \eqref{eq:rovnice}. However, this is not possible since $\divo \w \neq  0$. 
  Instead, we start by using Corollary~\ref{difOrlicz}, where $Q$ is replaced by $\tfrac 34 Q$, with the setting 
  \begin{equation*}
    f:=\divo(\xi^2\, \dhp \u), \quad h<\tfrac{1}{4} \ell(Q),\quad \psi =\phi.
  \end{equation*}
  It provides
  a~function $\vv \in W^{1,\phi}_{0}(\tfrac 34 Q)$ which solves
  \begin{equation*}
  \begin{aligned}
    \divo \vv &= \divo(\xi^2\, \dhp \u) &&\qquad \text{in }\tfrac 34 Q, \\
    \vv &=\bfzero &&\qquad  \text{on }\partial\tfrac 34 Q,
  \end{aligned}
  \end{equation*}
  and satisfies the corresponding estimates in
  Corollary~\ref{difOrlicz}. Therefore\footnote{We extend $\vv$ and $f$ by
    zero to the whole domain $\Omega$.}, the vector field  
	$\bw=$ \linebreak $\dhm\big (\xi ^2 \, \dhp\u -\vv\big )$ 
  belongs to $W^{1,\phi}_{0,\divo}(\Omega)$ and from \eqref{eq:rovnice} we get
\begin{equation}
  \label{eq:dtTx2}
  \begin{aligned}
    \intO& \xi^2\dhp{\T(\Du)}\cdot \dhp \Du \dx
    \\
    =&\intO  \T(\Du)\cdot \bD \dhm \vv -\T(\Du)\cdot{\dhm\big(2\xi \nabla \xi
      \otimess \dhp\u\big)} \dx
    \\
    &+\intO\f\cdot\dhm(\xi^2 \dhp \u) - \f \cdot \dhm \vv \dx=:\sum_{j=1}^{4} I_j.
  \end{aligned}
\end{equation}
The term providing the information concerning the regularity of the
solution is the integral on the left-hand side. Indeed, from \eqref{eq:3a} one gets
\begin{equation*}
    \intO \xi^2 \bigabs{ \dhp \bF (\Du) }^2\dx \le c \intO\xi ^2 \dhp{\T(\Du)}\cdot \dhp \Du \dx. 
\end{equation*}
In the same way as in \cite[Lemma 3.11]{br-reg-shearthin} one can show
that 
\begin{equation*}
  \begin{aligned}
    \intO \phi\big(\xi \abs{\nabla \dhpm\u}\big) \dx &\le c\,\intO
    \xi^2 \bigabs{ \dhpm \bF (\Du) }^2 \dx + c(
    \norm{\xi}_{1,\infty})
    \int_{Q} \phi \big (\abs{ \nabla\u }\big ) \dx. \hspace*{-3mm}
  \end{aligned}
\end{equation*}
Combining the last two inequalities together with \eqref{eq:dtTx2}, we obtain the inequality
\begin{equation}\label{eq:tang}
  \intO  \xi^2 \bigabs{ \dhp \bF (\Du) }^2  
          + \phi (\xi \abs{\nabla \dhp \u})\dx    \le c \,
          \sum_{j=1}^{4} |I_j|+ c(\norm{\xi}_{1,\infty}) \int_Q
          {\phi}(|\nabla \bu|)\dx.  
\end{equation}
Now we are going to find an~appropriate estimate for each of the terms $I_j$.
At first, observe that if $\bQ\in\R^{n\times n}_\sym$ and $t>0$, then \eqref{eq:hammere} (with $\bP=\mathbf 0$) and \eqref{eq:young:2} (recall that $\Delta_2(\phi)<\infty$) yield 
	\begin{equation}\label{eq:yg}
		|\bS(\bQ)|t \le c \, \phi'(|\bQ|) t \le c(\eps^{-1}) \phi(|\bQ|) + \eps \,\phi(t).
	\end{equation}
Making use of \eqref{eq:yg}, \eqref{eq:2} and the $\Delta_2$-condition for $\phi$, we have
\begin{align*}
  \bigabs{I_2}
  &\le c(\eps ^{-1}) \int_{Q}  \phi(\abs {\bD  \u}) \dx + \eps \int_{\Omega}
    \phi\big (\abs{\dhm{\big(2\xi \nabla \xi    \otimess \dhp\u\big)}}\big )\dx
  \\
  &\le c(\eps ^{-1}) \int_{Q}  \phi(\abs {\bD \u}) \dx +  \eps  \int_{\Omega}
    \phi\big (\abs{\nabla {\big(2\xi \nabla \xi \otimess \dhp\u\big)}}\big )\dx
  \\
  & \le c(\eps ^{-1},\norm{\xi}_{2,\infty}) \int_{Q}  \phi(\abs {\nabla
    \u}) \dx + \eps \, c(\norm{\xi}_{1,\infty})\int_{\Omega}
    \phi\big (\abs{\xi \nabla \dhp\u }\big )\dx.
\end{align*}
Next, inequality \eqref{eq:young:1}, the fact that $\Delta_2(\phi)<\infty$ and $0\le\xi^2\le\xi$,  imply
\begin{align*}
    \bigabs{I_3} &\le c(\eps ^{-1}) \int_{Q}  \phi^*(|\f|) \dx + \eps \int_{\Omega}
        \phi\big(|\dhm(\xi^2 \dhp \u)| \big)\dx
    \\
    &\le c(\eps ^{-1}) \int_{Q}  \phi^*(|\f|) \dx   +  \eps  \int_{\Omega}
        \phi\big( |\nabla(\xi^2 \dhp \u)|\big )\dx
    \\
    & \le c(\eps ^{-1}, \norm{\xi}_{1,\infty})\int_{\Omega} \phi^*(|\f|) +  \phi(|\nabla\u|) \dx + \eps \, c\int_{\Omega}
      \phi\big (\abs{\xi \nabla \dhp\u }\big )\dx.
\end{align*}
In order to use Corollary~\ref{difOrlicz} we observe that
\begin{align}
  \label{eq:10}
  \begin{aligned}
    \abs{\dhpm \divo(\xi^2 \dhp \u)} &\le c(\norm{\xi}_{2,\infty}) \,\xi \, \abs {\dhp
      \u} + c(\norm{\xi}_{1,\infty}) \abs{\dhpm(\xi \, \dhp \u)},
    \\
    |\divo(\xi^2 \dhp \u)| &\le c(\norm{\xi}_{1,\infty}) \,\xi \, \abs{\dhp \u},
  \end{aligned}
\end{align}
where we also used that $\u$ is solenoidal.  Estimate \eqref{eq:yg},
Corollary~\ref{difOrlicz}(iii), inequalities \eqref{eq:10},
\eqref{eq:2}, \eqref{eq:1a} and the condition $\Delta_2(\psi)<\infty$ now yield the following:
\begin{align*}
    \bigabs{I_1} &\le c(\eps ^{-1}) \int_{Q} \phi(\abs {\bD  \u}) \dx + \eps \int_{\Omega} \phi\big (\abs{\nabla
      {\dhm\vv}}\big )\dx
    \\
    &\le c(\eps ^{-1}) \int_{Q} \phi(\abs {\Du}) \dx \\
    & \qquad + \eps\,C_4 \int_{\Omega} \phi\big (|\dhp \divo(\xi^2 \dhp \u)|\big )
      +  \phi\big (|\dhm \divo(\xi^2 \dhp \u)|\big )+ \phi\big (|\divo(\xi^2 \dhp \u)|\big )\dx
    \\
    & \le c(\eps ^{-1},C_4, \norm{\xi}_{2,\infty}) \int_{Q} \phi(\abs {\nabla \u}) \dx \\
    & \qquad + \eps\, C_4\, c(\norm{\xi}_{1,\infty})\int_{\Omega} \phi\big (\abs{\dhp (\xi\dhp\u )}\big )+ \phi\big (\abs{\dhm (\xi
      \dhp\u )}\big )\dx
    \\
    & \le c(\eps^{-1},C_4,\norm{\xi}_{2,\infty}) \int_{Q}   \phi(\abs {\nabla \u}) \dx + \eps\, C_4\,
    c(\norm{\xi}_{1,\infty})\int_{\Omega} \phi\big (\xi \abs{\nabla 
      \dhp\u }\big )\dx.
\end{align*}
Finally, using inequality \eqref{eq:young:1} with
$\Delta_2(\phi)<\infty$, estimate \eqref{eq:2}, Corollary 
\ref{difOrlicz}(ii) and inequality \eqref{eq:10}, one shows that
\begin{align*}
    \bigabs{I_4} 
    &\le c \int_{Q} \phi^*(|\f|) \dx + c \int_{\Omega} \phi\big (|\dhm\vv| \big)\dx    \\
    &\le c \int_{Q} \phi^*(|\f|) \dx + c \int_{Q} \phi\big (|\nabla\vv| \big)\dx   \\
    &\le c \int_{Q} \phi^*(|\f|) \dx + c \int_{Q} \phi\big (|\divo(\xi^2\, \dhp \u)| \big)\dx\\
    &\le c \int_{Q} \phi^*(|\f|) + \phi\big (|\nabla\bu| \big) \dx.
\end{align*}
Applying the obtained estimates of $|I_j|$ to \eqref{eq:tang}, we get
	\begin{align*}
          &\intO  \xi^2 \bigabs{ \dhp \bF (\Du) }^2  + \phi (\xi
            \abs{\nabla \dhp \u})\dx
          \\
          &\le c(\eps^{-1},C_4, \norm{\xi}_{2,\infty}) \int_Q
            {\phi^*}(|\f|) +\phi\big (|\nabla\bu| \big) \dx +
            \eps\,c(C_4,\norm{\xi}_{1,\infty})\int_{\Omega} \phi\big
            (|\xi \nabla	\dhp\u|\big )\dx.   
	\end{align*}
By choosing $\eps $ sufficiently small, the last term on the
right-hand side may be absorbed into the left-hand side of the
inequality. Hence, we have proved 
\begin{align*}
  \intO & \xi^2 \bigabs{ \dhp \bF (\Du) }^2 \dx 
          \le  c(\norm{\xi}_{2,\infty}) \int_Q {\phi^*}(|\f|) +\phi\big (|\nabla\bu| \big)\dx.
\end{align*}
In view of \eqref{eq:2a}, \eqref{eq:2b}, this implies the desired estimate \eqref{eq:est-reg-F2}.
\end{proof}

\begin{acknowledgements}
  We would like to thank the referees for their helpful comments. 
\end{acknowledgements}


\def\cprime{$'$}
\ifx\undefined\bysame
\newcommand{\bysame}{\leavevmode\hbox to3em{\hrulefill}\,}
\fi


\end{document}